\documentclass[%
  a4paper,
  onecolumn,
  colorlinks,
]{preprint}

\usepackage{geometry}
\usepackage[normalem]{ulem}
\usepackage{graphicx} % Required for inserting images
\usepackage{amsmath}
\usepackage{amsthm}
\usepackage{amssymb}
\usepackage{xcolor}

\usepackage{algorithm}
\usepackage{algpseudocode}

\usepackage{subcaption}
\usepackage{enumitem}
\usepackage{pgfplots}
\pgfplotsset{compat = 1.9}

\algrenewcommand\algorithmicrequire{\textbf{Input:}}
\algrenewcommand\algorithmicensure{\textbf{Output:}}

\newcommand{\R}{\mathbb{R}}
\newcommand{\C}{\mathbb{C}}
\renewcommand{\H}{\mathbb{H}}
\renewcommand{\L}{\mathbb{L}}

\newcommand{\D}{\mathbb{D}}
\newcommand{\N}{\mathbb{N}}

\renewcommand{\Im}{\operatorname{Im}}
\newcommand{\rk}{\operatorname{rank}}

\DeclareMathOperator*{\argmax}{arg\,max}
\DeclareMathOperator*{\argmin}{arg\,min}
\DeclareMathOperator*{\diag}{diag}
\DeclareMathOperator*{\vectorize}{vec}

\newcommand{\bb}{\textbf{b}}
\newcommand{\bc}{\textbf{c}}
\newcommand{\bd}{\textbf{d}}
\newcommand{\be}{\textbf{e}}
\newcommand{\bff}{\textbf{f}}

\newcommand{\bk}{\textbf{k}}
\newcommand{\bp}{\textbf{p}}
\newcommand{\br}{\textbf{r}}
\newcommand{\bn}{\textbf{n}}
\newcommand{\bu}{\textbf{u}}
\newcommand{\bv}{\textbf{v}}
\newcommand{\bw}{\textbf{w}}
\newcommand{\bx}{\textbf{x}}
\newcommand{\by}{\textbf{y}}
\newcommand{\bz}{\textbf{z}}

\newcommand{\bA}{\textbf{A}}

\newcommand{\bC}{\textbf{C}}
\newcommand{\bD}{\textbf{D}}
\newcommand{\bE}{\textbf{E}}
\newcommand{\bF}{\textbf{F}}

\newcommand{\bH}{\textbf{H}}
\newcommand{\bI}{\textbf{I}}
\newcommand{\bJ}{\textbf{J}}
\newcommand{\bK}{\textbf{K}}
\newcommand{\bL}{\textbf{L}}

\newcommand{\bP}{\textbf{P}}
\newcommand{\bQ}{\textbf{Q}}
\newcommand{\bR}{\textbf{R}}
\newcommand{\bU}{\textbf{U}}
\newcommand{\bV}{\textbf{V}}
\newcommand{\bW}{\textbf{W}}
\newcommand{\bX}{\textbf{X}}
\newcommand{\bY}{\textbf{Y}}
\newcommand{\bZ}{\textbf{Z}}

\newcommand{\blambda}{\boldsymbol{\lambda}}

\newcommand{\hi}{\hat{\imath}}

\newcommand{\tn}{\tilde{n}}
\newcommand{\ti}{\tilde{\imath}}

\newcommand{\cC}{\mathcal{C}}

\newcommand{\cH}{\mathcal{H}}

\newcommand{\cO}{\mathcal{O}}

\newcommand{\cU}{\mathcal{U}}
\newcommand{\cX}{\mathcal{X}}
\newcommand{\cY}{\mathcal{Y}}

\let\originalleft\left
\let\originalright\right
\renewcommand{\left}{\mathopen{}\mathclose\bgroup\originalleft}
\renewcommand{\right}{\aftergroup\egroup\originalright}
\newcommand{\fronorm}[1]{\left\lVert #1 \right\rVert_{\operatorname{F}}}
\newcommand*\customstrut[1]{\vrule width0pt height0pt depth#1\relax}
\newtheorem{lemma}{Lemma}

\definecolor{darkgreen}{RGB}{0, 90, 10}
\definecolor{darkblue}{RGB}{53, 42, 135}
\definecolor{orangeyellow}{RGB}{225, 185, 82}
\definecolor{greenish}{RGB}{89, 189, 140}

\tikzstyle{sline} = [
  solid,
  line width = 1.5pt
]

\pgfplotscreateplotcyclelist{trigLSerrs}{
  {purple, sline},
  {black, sline},
}

\pgfplotscreateplotcyclelist{newpoint}{
  {black, sline},
}

\pgfplotscreateplotcyclelist{synthetic}{
  {darkblue, sline},
  {orangeyellow, sline},
  {purple, sline},
  {greenish, sline},
}

\pgfplotsset{
    colormap={paaacmap}{
        rgb255=(53, 42, 135)   % Dark blue
        rgb255=(15, 92, 221)   % Medium blue
        rgb255=(18, 125, 216)  % Lighter blue
        rgb255=(7, 156, 207)   % Cyan-like
        rgb255=(21, 177, 180)  % Teal
        rgb255=(89, 189, 140)  % Greenish
        rgb255=(165, 190, 107) % Yellowish-green
        rgb255=(225, 185, 82)  % Orange-yellow
        rgb255=(252, 206, 46)  % Bright yellow
        rgb255=(249, 251, 14)  % Vibrant yellow
    }
}

\begin{document}
	
    \title{Multivariate Rational Approximation via Low-Rank Tensors and the p-AAA Algorithm}
    
    \author[$\ast$]{Linus Balicki}
    \author[$\ast$]{Serkan Gugercin}
    \affil[$\ast$]{Department of Mathematics, Virginia Tech, Blacksburg, VA, 24061, USA}
              
    
    \abstract{Approximations based on rational functions are widely used  in various applications across computational science and engineering. For univariate functions, the adaptive Antoulas-Anderson algorithm (AAA), which uses the barycentric form of a rational approximant,  has established itself as a powerful tool for efficiently computing such approximations. The p-AAA algorithm, an extension of the AAA algorithm specifically designed to address multivariate approximation problems, has been recently introduced.
    A common challenge in multivariate approximation methods is that multivariate problems with a large number of variables often pose significant memory and computational demands. To tackle this hurdle in the setting of p-AAA, we first introduce  barycentric forms that are represented in the terms of separable functions. This then leads to the low-rank p-AAA algorithm which leverages low-rank tensor decompositions in the setting of barycentric rational approximations. We discuss various theoretical and practical aspects of the proposed computational framework and showcase its effectiveness on four numerical examples. We focus specifically on applications in parametric reduced-order modeling for which higher-dimensional data sets can be tackled effectively with our novel procedure. }
    
    \novelty{}
    
    \maketitle
	
	\section{Introduction}
	Approximating a function from a given set of its samples is a fundamental step in many applications across computational science and engineering. In particular, rational function-based approximation frameworks have been extensively studied and applied in both theoretical and practical contexts. The appeal of rational approximations largely arises from their effectiveness in capturing a function's behavior near singularities or on unbounded domains \cite{trefethen_approximation_2020}. This fact has made rational functions ubiquitous in various fields ranging from model reduction to system identification to PDE approximation theory to numerical linear algebra and many more. For a selected few references highlighting such diverse applications of rational functions, we refer the reader to, e.g., \cite{ho_effective_1966,silverman_realization_1971,BerG17,DramacGB2015QuadVF,baur2014model,gosea2024structured,druskin2009solution,borcea2021reduced,bertram_class_2024,benner2013numerical,simoncini2016computational,wilber2021computing,antoulas_approximation_2005,guttel_nonlinear_2017,antoulas_interpolatory_2020,gustavsen_rational_1999,karachalios_6_2021,boulle_data-driven_2022,brubeck_lightning_2022,wilber_data-driven_2022,rubin_bounding_2022,austin_computing_2015,lietaert_automatic_2022,brennan_contour_2023}
    and the references therein. 
    %\cite{antoulas_approximation_2005, beattie_chapter_2017, antoulas_interpolatory_2020}, system identification \cite{gustavsen_rational_1999,karachalios_6_2021}, PDE approximation \cite{boulle_data-driven_2022,brubeck_lightning_2022} and numerical analysis \cite{austin_computing_2015,lietaert_automatic_2022,brennan_contour_2023}. 

    Rational approximants can be equivalently represented in many different forms.  A numerically robust and theoretically elegant representation is the so-called barycentric form of rational approximants~\cite{berrut_barycentric_2004}.
    For univariate functions $f:\C \rightarrow \C$ many effective rational approximation frameworks have been developed using the barycentric form; see, e.g., \cite{antoulas_scalar_1986,berrut_barycentric_2004,cirillo2019advances,nakatsukasa_aaa_2018,antoulas_chapter_2017,gustavsen_rational_1999,driscoll_aaa_2024} and the references therein. Recently, several efforts have been made to extend these barycentric form based methods to the multivariate setting where we aim to approximate a function $\bff:\C^d \rightarrow \C$ by a multivariate rational function $\br:\C^d \rightarrow \C$; see, e.g., \cite{deschrijver_stability_2008,antoulas_two-variable_2012,ionita_data-driven_2014,rodriguez_p-aaa_2023}.  In order to make our introduction more concise we will focus on the two-variable case for now and generalize the discussion in the proceeding sections.
    In the two-variable case ($d=2$) with $\bz = \left(z^{(1)},z^{(2)}\right) \in \C^2$,
    the barycentric form of the rational approximant reads
	\begin{equation}
		\label{eq:introbarycentricform}
		\br\left(\bz \right) = \sum_{i_1=1}^{n_1} \sum_{i_2=1}^{n_2} \frac{\alpha_{i_1 i_2} \bff\left(\lambda^{(1)}_{i_1},\lambda^{(2)}_{i_2}\right)}{\left(z^{(1)}-\lambda^{(1)}_{i_1}\right) \left(z^{(2)}-\lambda^{(2)}_{i_2}\right)} \Bigg/ \sum_{i_1=1}^{n_1} \sum_{i_2=1}^{n_2} \frac{\alpha_{i_1 i_2} }{\left(z^{(1)}-\lambda^{(1)}_{i_1}\right) \left(z^{(2)}-\lambda^{(2)}_{i_2}\right)}.
	\end{equation}
     We will discuss properties of the introduced barycentric form~\eqref{eq:introbarycentricform} in detail in Section~\ref{sec:multivariatebarycentric}, but we point out the key components here briefly: 
	As in~\eqref{eq:introbarycentricform} and throughout this paper we will use superscripts to indicate that a variable, matrix or set is associated with the $j$-th variable $z^{(j)}$ and subscripts to specify the entry of a matrix or index of an element belonging to a set. In~\eqref{eq:introbarycentricform}    
    $\lambda^{(1)}_{i_1}, \lambda^{(2)}_{i_2} \in \C$ with $(i_1,i_2) \in \{ 1, \ldots, n_1\} \times \{ 1, \ldots, n_2 \}$ 
    are the barycentric nodes with $\lambda^{(1)}_{i_1}$ being associated to the variable $z^{(1)}$ and $\lambda^{(2)}_{i_2}$ to $z^{(2)}$. 
    Furthermore, $\alpha \in \C^{n_1 \times n_2}$ is the matrix of barycentric coefficients $\alpha_{i_1i_2}$. As we will discuss in more detail in Section~\ref{sec:multivariatebarycentric},
    $\br\left(\bz \right)$ is a two-variable rational function of
    $z^{(1)}$ and $z^{(2)}$. Quality of the rational approximant 
    $\br\left(\bz \right)$ solely depends on good choices
    of the barycentric nodes and coefficients, and determining such good choices is at the core of many rational approximation algorithms such as the parametric Loewner framework \cite{ionita_data-driven_2014,antoulas_loewner_2024}. Another recently proposed approach which falls into this category     
    is the parametric adaptive Antoulas–Anderson (p-AAA) algorithm \cite{rodriguez_p-aaa_2023} which extends the AAA framework \cite{nakatsukasa_aaa_2018} to the multivariate setting. The p-AAA framework is our main focus in this paper. The p-AAA algorithm is a data-driven approximation method which leverages a given set of samples
	\begin{equation}
		\label{eq:samples}
		\bF = \left\{ \bff(\bz) \; \bigg| \; \bz \in \bZ^{(1)} \times \bZ^{(2)} \right\} \subset \C
	\end{equation}
	with corresponding sampling points
	\begin{equation}
		\label{eq:samplingvalues}
		\bZ^{(1)} = \left\{Z^{(1)}_1,\ldots,Z^{(1)}_{N_1}\right\} \subset \C \quad \text{and} \quad \bZ^{(2)} = \left\{Z^{(2)}_1,\ldots,Z^{(2)}_{N_2}\right\} \subset \C,
	\end{equation}
	to construct the rational approximant $\br \approx \bff$ in the barycentric form introduced in \eqref{eq:introbarycentricform}. As in the univariate case of AAA, the two central steps of p-AAA are (i) choosing barycentric nodes via a greedy selection and (ii) determining the barycentric coefficient matrix $\alpha$ by solving a linear least-squares (LS) problem of the form
	\begin{equation}
		\label{eq:introls}
		\min_{\fronorm{\alpha}=1}\lVert \L_2 \vectorize(\alpha) \rVert_2^2.
	\end{equation}
	The optimization problem~\eqref{eq:introls} has a closed-form solution in terms of the singular value decomposition (SVD) of the higher-order Loewner matrix $\L_2 \in \C^{N_1 N_2 \times n_1 n_2}$, which we will introduce later. Computing the SVD of $\L_2$ is the dominant cost of the algorithm. While solving the p-AAA LS problem efficiently is typically feasible in the two-variable setting, the computational complexity of the algorithm increases drastically when moving to many variables.  This is primarily due to the fact that the barycentric form will be based on a higher-order tensor of barycentric coefficients $\alpha \in \C^{n_1 \times \cdots \times n_d}$ when considering $d > 2$ variables. In this case, the computational complexity of the underlying LS problem grows exponentially with the number of added variables as in many multivariate approximation problems. This is typically referred to as the ``curse of dimensionality" and frequently affects algorithms for modeling and approximating multivariate functions. A common approach to overcome the computational burden associated with many variables is to consider low-rank tensor approximations which are closely connected to representations utilizing separable functions \cite{grasedyck_literature_2013}. A key aspect of this paper is to investigate how this perspective connects to multivariate barycentric forms and rational approximation with the p-AAA algorithm. In this context, we will show how to utilize low-rank tensor decompositions for the barycentric coefficients $\alpha$ to significantly alleviate the computational burden imposed by the curse of dimensionality in the p-AAA framework. This will lead to a modified formulation of the algorithm which is scalable well beyond the two variable case. All throughout, we consider applications in parametric reduced-order modeling as a motivation for our analysis and proposed algorithm. 
	
	The next sections are structured as follows: Section~\ref{sec:motivatingproblems} discusses two fundamental  problems in data-driven reduced order modeling which will serve as our primary motivation for investigating multivariate rational approximation algorithms. In Section~\ref{sec:pAAA} we revisit multivariate barycentric forms and the p-AAA algorithm. We introduce a revised version of the formulations established in previous works,  resulting in a concise description of the 
    many-variable setting. In Section~\ref{sec:lrbarycentric} we introduce barycentric forms that are represented in terms of separable functions and establish their connection to low-rank tensor decompositions. Section~\ref{sec:lrpaaa} comprises our main result and develops the low-rank p-AAA algorithm. In Section~\ref{sec:practical} we discuss important practical aspects of our proposed method. Several numerical examples are presented in Section~\ref{sec:numericalexamples}. Finally, conclusions and possible future research directions are summarized in Section~\ref{sec:conclusion}.
	
	\section{Motivating Problems}
	\label{sec:motivatingproblems}
	We begin by considering two fundamental problems arising in reduced order modeling. As we will see, both of these problems aim to approximate a function of the form
	\begin{equation}
		\label{eq:fstructure}
		\bff(\bz) = \bc(\bz)^\top \bK(\bz)^{-1} \bb(\bz) \in \C,
	\end{equation}
	where $\bb(\bz),\bc(\bz) \in \C^n$ and $\bK(\bz) \in \C^{n \times n}$. We point out that in many cases $\bff$ itself will be a rational function (e.g., if $\bb,\bc,\bK$ have a rational dependence on $\bz$), which makes rational approximation appealing. However, even in cases where this is not satisfied, rational functions are typically still good choices for approximating $\bff$ due to singularities arising from the $\bK(\bz)^{-1}$ term. \\
	
	First, we consider a linear stationary parameter-dependent partial differential equation (PDE) in the weak form
	\begin{align*}
		a(\bv(\bp),\bw;\bp) &= g(\bw;\bp) \quad \text{for all } \bw \in \cH, \\
		\by(\bp) &= l(\bv(\bp);\bp),
	\end{align*}
	where $\bp\in \C^d$ is a vector of parameters, $a(\cdot,\cdot;\bp) \in \R$ is a bilinear form, $\bv(\bp) \in \mathcal{H}$ is the solution in the real Hilbert space $\cH$, $\by(\bp) \in \R$ is the model output, and $g(\cdot;\bp),l(\cdot;\bp) \in \R$ are linear functionals. We assume that all quantities are appropriately defined such that existence and uniqueness of solutions is guaranteed. After discretizing the problem (e.g., via the finite element method) we obtain a model of the form
	\begin{equation}
		\label{eq:stationaryparametric}
		\begin{aligned}
			\bA(\bp) \bx(\bp) &= \bb(\bp), \\
			\bff(\bp) &= \bc(\bp)^\top \bx(\bp),
		\end{aligned}
	\end{equation}
	for which the output $\bff$ can be written in the form $\bff(\bp) = \bc(\bp)^\top \bA(\bp)^{-1} \bb(\bp)$ resembling the structure introduced in \eqref{eq:fstructure}. Typically, the dimension of the discretized space and thus the dimension of the matrix $\bA(\bp)$ is very large (possibly in the millions) and evaluating $\bff$ exactly for many different parameters $\bp$ is computationally demanding. Therefore, computationally efficient approximations of $\bff$ are crucial for solving many-query tasks such as design optimization, inverse problems or uncertainty quantification. In particular, prior simulations or real-world experiments may provide samples of $\bff$, enabling the application of data-driven rational approximation algorithms to compute fast, yet accurate reduced order models in terms of a rational function $\br \approx \bff$.  \\
	
	For our second problem we consider a parametrized linear time-invariant (LTI) system of the form
	\begin{equation}
		\label{eq:lti}
		\begin{aligned}
			\bE(\bp)\dot{\bx}(t,\bp) &= \bA(\bp) \bx(t,\bp) + \bb(\bp) \bu(t), \\
			\by(t,\bp) &= \bc(\bp)^\top \bx(t,\bp), 
		\end{aligned}
	\end{equation}
	where $\bu(t) \in \R$ is the input, $\by(t,\bp) \in \R$ is the output and $\bx(t,\bp) \in \R^n$ is the state of the system which depends on the parameter vector $\bp \in \C^{d-1}$. Taking the Laplace transformation of the state and output equation yields
	\begin{align}
        \left( s \bE(\bp) - \bA(\bp) \right) \cX(s,\bp) &= \bb(\bp) \cU(s), \label{eq:laplacestateeq} \\
        \cY(s,\bp) &= \bc(\bp)^\top \cX(s,\bp), \label{eq:laplaceoutputeq}
	\end{align}
	where $s\in\C$ is the Laplace variable and $\cU,\cY,\cX$ are the Laplace transformations of the input, output and state, respectively. We introduce the transfer function of the LTI system in \eqref{eq:lti} as
	\begin{equation}
		\label{eq:tf}
		\bH(s,\bp) = \bc(\bp)^\top (s \bE(\bp) - \bA(\bp))^{-1} \bb(\bp),
	\end{equation}
	and note that it fully describes the input output behavior in the Laplace domain via $\cY(s,\bp) = \bH(s,\bp) \cU(s)$. This can be seen by rearranging the Laplace-transformed state equation in \eqref{eq:laplacestateeq} and plugging it into the corresponding output equation in \eqref{eq:laplaceoutputeq}. The corresponding data-driven reduced order modeling problem consists of approximating $\bH$ based on a given set of transfer function samples originating from real-world experiments or numerical computations. Similar to the stationary model, we observe that $\bH$ obeys the form introduced for $\bff$ in \eqref{eq:fstructure} by defining $\bz = (s,\bp) \in \C^d$. \\
	
	We will provide specific examples for these problem settings in the numerical examples presented in Sections~\ref{sec:synthetic}, \ref{sec:thermalblock} and \ref{sec:massspringdamper}. For now, we will consider the problems described here as a baseline motivation for using rational approximation algorithms, but emphasize that there is a much wider range of problems in which multivariate rational approximation is applicable and effective. 
	
	\section{The p-AAA Algorithm for Multivariate Rational Approximation}
	\label{sec:pAAA}
	We begin this section by revisiting the properties of barycentric forms for the two variable ($d=2$) case as introduced in \eqref{eq:introbarycentricform}. Afterwards, we discuss the p-AAA algorithm for this setting. Finally, we show how p-AAA can be extended to the many-variable setting, i.e., $d>2$. 
    The Loewner matrices play a crucial role in p-AAA. 
    In this section we introduce a new representation for these higher-order Loewner matrix, which generalizes an earlier result from \cite{ionita_data-driven_2014} as we explain in detail below. This new representation allows for implementing higher-order Loewner matrices in an efficient and straightforward manner and plays a fundamental role in our proposed method which we introduce in Section~\ref{sec:lrpaaa}.

	\subsection{Multivariate Barycentric Forms}
    \label{sec:multivariatebarycentric}
	Various types of barycentric forms are used in univariate polynomial \cite{berrut_barycentric_2004} and rational \cite{berrut_lebesgue_1997} interpolation frameworks. In the multivariate setting, barycentric forms of rational functions are the backbone of several approximation algorithms as well \cite{deschrijver_stability_2008,antoulas_two-variable_2012,ionita_data-driven_2014,rodriguez_p-aaa_2023}. The success of barycentric forms is primarily explained by the fact that they allow for efficient and numerically stable function evaluation and approximation. To be able to motivate the ideas more clearly and concisely for the  many-variable case that we are interested in, here we first focus on the two-variable barycentric form introduced in \eqref{eq:introbarycentricform}, which we write as
	\begin{equation*}
		\br(\bz) = \frac{\bn(\bz)}{\bd(\bz)},
	\end{equation*}
	where
	\begin{equation}
		\label{eq:singularnd}
		\begin{aligned}
			\bn(\bz) &:= \sum_{i_1=1}^{n_1} \sum_{i_2=1}^{n_2} \frac{\alpha_{i_1 i_2} \bff\left(\lambda^{(1)}_{i_1},\lambda^{(2)}_{i_2}\right)}{\left(z^{(1)}-\lambda^{(1)}_{i_1}\right) \left(z^{(2)}-\lambda^{(2)}_{i_2}\right)}, \\
			\bd(\bz) &:= \sum_{i_1=1}^{n_1} \sum_{i_2=1}^{n_2} \frac{\alpha_{i_1 i_2} }{\left(z^{(1)}-\lambda^{(1)}_{i_1}\right) \left(z^{(2)}-\lambda^{(2)}_{i_2}\right)}.
		\end{aligned}
	\end{equation}
	We note that the functions $\bn$ and $\bd$ defined above are not polynomials but rather rational functions themselves. In order to write $\br$ as a ratio of two polynomials we introduce the multivariate Lagrange polynomials
	\begin{align}
		\bP(\bz) &= \prod_{i_1=1}^{n_1} \prod_{i_2=1}^{n_2} \left(z^{(1)}-\lambda^{(1)}_{i_1}\right) \left(z^{(2)}-\lambda^{(2)}_{i_2} \right) \bn(\bz) \notag \\
		&= \sum_{i_1=1}^{n_1} \sum_{i_2=1}^{n_2} \prod_{\hi_1 \neq i_1} \prod_{\hi_2 \neq i_2} \alpha_{i_1 i_2} \left(z^{(1)}-\lambda^{(1)}_{\hi_1}\right) \left(z^{(2)}-\lambda^{(2)}_{\hi_2}\right) \bff\left(\lambda^{(1)}_{i_1},\lambda^{(2)}_{i_2}\right), \quad \text{and} \notag \\
		\bQ(\bz) &= \prod_{i_1=1}^{n_1} \prod_{i_2=1}^{n_2} \left(z^{(1)}-\lambda^{(1)}_{i_1}\right) \left(z^{(2)}-\lambda^{(2)}_{i_2} \right) \bd(\bz) \label{eq:largrangedenom} \\
		&= \sum_{i_1=1}^{n_1} \sum_{i_2=1}^{n_2} \prod_{\hi_1 \neq i_1} \prod_{\hi_2 \neq i_2} \alpha_{i_1 i_2} \left(z^{(1)}-\lambda^{(1)}_{\hi_1}\right) \left(z^{(2)}-\lambda^{(2)}_{\hi_2}\right), \notag
	\end{align}
	which allows for writing
	\begin{equation}
		\label{eq:polyrepresentation}
		\br(\bz) = \frac{\bP(\bz)}{\bQ(\bz)}.
	\end{equation}
	This reveals that $\br$ is a ratio of two polynomials with degree $n_1-1$ in $z^{(1)}$ and $n_2-1$ in $z^{(2)}$. With the additional assumption that $\alpha_{i_1 i_2} \neq 0$ for $(i_1,i_2) \in \{ 1, \ldots, n_1\} \times \{ 1, \ldots, n_2 \}$, we have that $\br$ is a proper rational function. We call the tuple $(n_1-1,n_2-1)$ the order of $\br$. Additionally, the barycentric form in \eqref{eq:polyrepresentation} is interpolatory as it satisfies $\br(\blambda) = \bff(\blambda)$ for $\blambda = \left(\lambda^{(1)}_{i_1},\lambda^{(2)}_{i_2}\right) \in \blambda^{(1)} \times \blambda^{(2)}$, where
	\begin{equation}
		\label{eq:interpolationnodes}
		\blambda^{(1)} = \left\{ \lambda^{(1)}_1, \ldots, \lambda^{(1)}_{n_1} \right\} \subset \bZ^{(1)} \quad \text{and} \quad \blambda^{(2)} = \left\{ \lambda^{(2)}_1, \ldots, \lambda^{(2)}_{n_2} \right\} \subset \bZ^{(2)}.
	\end{equation}
	This property can easily be shown by noting that 
	\begin{align*}
		\bP\left(\lambda^{(1)}_{i_1},\lambda^{(2)}_{i_2}\right) &= \prod_{\hi_1 \neq i_1} \prod_{\hi_2 \neq i_2} \alpha_{i_1 i_2} \left(\lambda^{(1)}_{i_1} -\lambda^{(1)}_{\hi_1}\right) \left(\lambda^{(2)}_{i_2} -\lambda^{(2)}_{\hi_2}\right) \bff\left(\lambda^{(1)}_{i_1},\lambda^{(2)}_{i_2}\right), \\
		\bQ\left(\lambda^{(1)}_{i_1},\lambda^{(2)}_{i_2}\right) &= \prod_{\hi_1 \neq i_1} \prod_{\hi_2 \neq i_2} \alpha_{i_1 i_2} \left(\lambda^{(1)}_{i_1} -\lambda^{(1)}_{\hi_1}\right) \left(\lambda^{(2)}_{i_2} -\lambda^{(2)}_{\hi_2}\right),
	\end{align*}
	and thus $\br(\blambda) = \bP(\blambda) / \bQ(\blambda) = \bff(\blambda)$. Therefore, the barycentric nodes in $\blambda^{(1)} \times \blambda^{(2)}$ are interpolation nodes for the representation in \eqref{eq:introbarycentricform} and the sets of interpolation nodes for each variable are defined in \eqref{eq:interpolationnodes}. \\
	
	We note that strictly speaking the representations of $\br$ in \eqref{eq:introbarycentricform} and $\eqref{eq:polyrepresentation}$ are not equal everywhere in $\C^2$ as the form in \eqref{eq:introbarycentricform} has removable singularities at the interpolation points while the form in \eqref{eq:polyrepresentation} does not. We treat this discrepancy explicitly by slightly rewriting the barycentric form in \eqref{eq:introbarycentricform}. The rewritten form will also be useful for our discussions performed in proceeding sections. Towards this goal we first introduce the modified Cauchy matrix $\cC\left(\blambda^{(j)},\bZ^{(j)}\right) \in \C^{n_j \times N_j}$ as
	\begin{equation}
		\cC\left(\blambda^{(j)},\bZ^{(j)}\right)_{ik} := \begin{cases}
			\frac{1}{Z^{(j)}_k - \lambda^{(j)}_i} \quad &\text{if } Z^{(j)}_k \notin \blambda^{(j)} \\
			1  \quad &\text{if } Z^{(j)}_k = \lambda^{(j)}_i \\
			0  \quad &\text{else }
		\end{cases}
        \label{eq:modifiedCauchy}
	\end{equation}
	where $\bZ^{(j)} = \left\{ Z^{(j)}_1,\ldots,Z^{(j)}_{N_j}\right\} \subset \C$ and $\blambda^{(j)} = \left\{ \lambda^{(j)}_{1},\ldots,\lambda^{(j)}_{n_j} \right\} \subset \C$. Note that for a scalar input $x \notin \blambda^{(j)}$, we have 
    \begin{equation*}
        \cC\left(\blambda^{(j)},x\right) = \left[ \frac{1}{x - \lambda^{(j)}_1}, \ldots, \frac{1}{x - \lambda^{(j)}_{n_j}} \right]^\top \in \C^{n_j}.
    \end{equation*}
    and
    \begin{equation*}
        \cC\left(\blambda^{(j)},\blambda^{(j)}\right) = \bI_{n_j}.
    \end{equation*}
    Moreover, if $\blambda^{(j)} \cap \bZ^{(j)} = \emptyset$, then $\cC(\blambda^{(j)},\bZ^{(j)})$ is a standard Cauchy matrix. In general, $\cC(\blambda^{(j)},\bZ^{(j)})$ will be a Cauchy matrix where some rows have been replaced by unit vectors. In order to ease notation in the following, we define 
	\begin{equation*}
		\cC^{(j)} := \cC\left(\blambda^{(j)},\bZ^{(j)}\right) \quad \text{and} \quad \cC^{(j)}\left(x\right) := \cC\left(\blambda^{(j)},x\right).
	\end{equation*}
	Further, for a given function $\bff$, we introduce the data matrix $\D \in \C^{N_1 \times N_2}$ whose entries are given by
	\begin{equation*}
		\D_{i_1 i_2} = \bff\left(Z^{(1)}_{i_1}, Z^{(2)}_{i_2}\right) \quad \text{for} \quad  (i_1,i_2) \in \{ 1, \ldots, N_1\} \times \{ 1, \ldots, N_2 \}.
	\end{equation*}
	 Similarly, we define the matrix of interpolated data     
     at as $\H \in \C^{n_1 \times n_2}$ via 
    \begin{equation*}
        \H_{i_1 i_2} = \bff\left(\blambda^{(1)}_{i_1},\blambda^{(2)}_{i_2}\right) \quad \text{for} \quad  (i_1,i_2) \in \{ 1, \ldots, n_1\} \times \{ 1, \ldots, n_2 \}.
    \end{equation*}
    We now can write the barycentric form as a ratio of two inner products
	\begin{equation}
		\label{eq:barycentricform}
		\br(\bz) = \left[ \cC^{(1)}\left(z^{(1)}\right) \otimes \cC^{(2)}\left(z^{(2)}\right)\right]^\top \vectorize(\alpha \circ \H) \bigg/ \left[\cC^{(1)}\left(z^{(1)}\right) \otimes \cC^{(2)}\left(z^{(2)}\right)\right]^\top \vectorize(\alpha),
	\end{equation}
	where $\circ$ denotes the element-wise (Hadamard) product  and $\vectorize(\alpha) \in \C^{n_1 n_2}$ denotes the row-wise vectorization of the matrix $\alpha \in  \C^{n_1 \times n_2}$, i.e.,
     \begin{equation*}
         \vectorize(\alpha)_{n_2 (i_1 - 1) + i_2} = \alpha_{i_1 i_2} \quad \text{for} \quad i_1 = 1,\ldots,n_1 \quad \text{and} \quad i_2 = 1,\ldots,n_2.
     \end{equation*}
    This representation for $\br$ treats the removable singularities explicitly and is thus equal to the representation in \eqref{eq:polyrepresentation} everywhere in $\C^2$. Since this representation will be more useful in our upcoming discussion we redefine the functions $\bn$ and $\bd$ previously introduced in \eqref{eq:singularnd} via
	\begin{align}
		\label{eq:nd}
		\bn(\bz) &:= \left[ \cC^{(1)}\left(z^{(1)}\right) \otimes \cC^{(2)}\left(z^{(2)}\right)\right]^\top \vectorize(\alpha \circ \H) \quad \text{and} \\
		\label{eq:dd}
		\bd(\bz) &:= \left[\cC^{(1)}\left(z^{(1)}\right) \otimes \cC^{(2)}\left(z^{(2)}\right)\right]^\top \vectorize(\alpha).
	\end{align}
	
	\subsection{The p-AAA Algorithm}
	At its core, the p-AAA algorithm aims to compute a barycentric form as in $\eqref{eq:barycentricform}$ which approximates the data given in \eqref{eq:samples} by choosing appropriate sets of interpolation nodes $\blambda^{(1)},\blambda^{(2)}$ and the matrix of barycentric coefficients $\alpha$. As in the univariate AAA case, in the p-AAA algorithm these choices are updated in an iterative manner. The algorithm starts with no interpolation points and initializes $\br \equiv \operatorname{average}(\D)$ as a constant function. Then, p-AAA successively updates the interpolation sets $\blambda^{(1)}$ and $\blambda^{(2)}$ by adding new points via a greedy selection. This means in each iteration the algorithm evaluates the current approximant $\br$ over all sampling points to compute the maximum error
	\begin{equation}
		\label{eq:greedy}
		\left(\lambda^{(1)}_*,\lambda^{(2)}_*\right) = \argmax_{\substack{\bz \in \bZ^{(1)} \times \bZ^{(2)}}} \lvert \br(\bz) - \bff(\bz) \rvert.
	\end{equation}
	The interpolation points are then updated via
	\begin{equation*}
		\blambda^{(1)} \leftarrow \blambda^{(1)} \cup \left\{\lambda^{(1)}_*\right\} \qquad \text{and} \qquad \blambda^{(2)} \leftarrow \blambda^{(2)} \cup \left\{\lambda^{(2)}_*\right\}.
	\end{equation*}
	Note that the interpolation set for at least one variable is guaranteed to increase unless all data is fully interpolated. Hence, in each iteration of p-AAA the order of the rational approximant $\br$ increases. Aside from the interpolation points, p-AAA chooses $\alpha$ such that $\br$ approximates samples in $\bF$ which are not interpolated by solving a linear LS problem. We derive the LS problem in the following paragraph. \\

    A typical approach for approximating the data in $\bF$ by $\br$ would be to choose $\alpha$ such that the LS error
    \begin{equation}
        \label{eq:nonlinearLSerror}
        \widetilde{\bE}(\alpha) = \sum_{i_1=1}^{N_1} \sum_{i_2=1}^{N_2} \left( \bff\left(Z^{(1)}_{i_1},Z^{(2)}_{i_2}\right) - \br\left(Z^{(1)}_{i_1},Z^{(2)}_{i_2}\right) \right)^2
    \end{equation}
    is minimized (assuming that the barycentric nodes are fixed). Let us consider the approximation error at an arbitrary sample $\left(Z^{(1)},Z^{(2)}\right) \in \bZ^{(1)} \times \bZ^{(2)}$ in the data set
	\begin{align*}
		\bff\left(Z^{(1)},Z^{(2)}\right) - \br\left(Z^{(1)},Z^{(2)}\right) &= \frac{\bd\left(Z^{(1)},Z^{(2)}\right)}{\bd\left(Z^{(1)},Z^{(2)}\right)}\bff\left(Z^{(1)},Z^{(2)}\right) - \frac{\bn\left(Z^{(1)},Z^{(2)}\right)}{\bd\left(Z^{(1)},Z^{(2)}\right)} \\
		&= \frac{1}{\bd\left(Z^{(1)},Z^{(2)}\right)} \left( \bd\left(Z^{(1)},Z^{(2)}\right) \bff\left(Z^{(1)},Z^{(2)}\right) - \bn\left(Z^{(1)},Z^{(2)}\right) \right).
	\end{align*}
 We observe that the expression above and therefore also the LS error $\widetilde{\bE}(\alpha)$ have a nonlinear dependence on the entries of $\alpha$ due to the $1 / \bd\left(Z^{(1)},Z^{(2)}\right)$ term. Thus, the objective function in \eqref{eq:nonlinearLSerror} is nonlinear in the sought after barycentric coefficients. In order to avoid solving an expensive nonlinear LS problem, p-AAA instead minimizes an error expression which is linearized by dropping the $1 / \bd\left(Z^{(1)},Z^{(2)}\right)$ terms. The linearized LS error expression over all data thus reads
	\begin{align}
		\bE(\alpha) &= \sum_{i_1=1}^{N_1} \sum_{i_2=1}^{N_2} \left( \bd\left(Z^{(1)}_{i_1},Z^{(2)}_{i_2}\right) \bff\left(Z^{(1)}_{i_1},Z^{(2)}_{i_2}\right) - \bn\left(Z^{(1)}_{i_1},Z^{(2)}_{i_2}\right) \right)^2 \label{eq:lserrorsum} \\
		&= \sum_{i_1=1}^{N_1} \sum_{i_2=1}^{N_2} \bigg( \left[\cC^{(1)}\left(Z^{(1)}_{i_1}\right) \otimes \cC^{(2)}\left(Z^{(2)}_{i_2}\right)\right]^\top \vectorize(\alpha) \bff\left(Z^{(1)}_{i_1},Z^{(2)}_{i_2}\right) \nonumber \\
        & \hspace{6cm} - \left[\cC^{(1)}\left(Z^{(1)}_{i_1}\right) \otimes \cC^{(2)}\left(Z^{(2)}_{i_2}\right)\right]^\top \vectorize(\alpha \circ \H) \bigg)^2 \nonumber \\
		&= \left\lVert \vectorize(\D) \circ \left( \left[\cC^{(1)} \otimes \cC^{(2)}\right]^\top \vectorize(\alpha) \right) - \left[\cC^{(1)} \otimes \cC^{(2)}\right]^\top \vectorize(\alpha \circ \H) \right\rVert_2^2 \nonumber \\
		&= \bigg\lVert \bigg( \underbrace{ \customstrut{3ex} \diag\left(\vectorize\left(\D\right)\right) \left[\cC^{(1)} \otimes \cC^{(2)}\right]^\top - \left[\cC^{(1)} \otimes \cC^{(2)}\right]^\top \diag(\vectorize(\H))}_{=:\L_2 \in \C^{N_1 N_2 \times n_1 n_2}} \bigg) \vectorize(\alpha) \bigg\rVert_2^2. \label{eq:L2}
	\end{align}
	The above equation shows that the linearized LS error can be written concisely as $\bE(\alpha) = \lVert \L_2 \vectorize(\alpha) \rVert_2^2$, where $\L_2$ is the $2$-d Loewner matrix. We note that our definition of $\L_2$ slightly deviates from the one used in \cite{rodriguez_p-aaa_2023} as the Loewner matrix introduced here has $n_1 n_2$ additional rows corresponding to the error at interpolated values in $\blambda^{(1)} \times \blambda^{(2)}$. These rows are zero and therefore do not effect the error $\bE(\alpha)$. We note that we could recover the 2-D Loewner matrix 
    from~\cite{rodriguez_p-aaa_2023} by removing the zero rows. However, the newly introduced representation of the Loewner matrix comes in handy in various ways: The new definition of $\L_2$ has a simple extension to several variables as outlined in Section~\ref{sec:paaamanyparameters}. Additionally, our new formulation of the Loewner matrix can easily be implemented in a few lines of code even for the many variable case. Prior work has not made it clear how to elegantly construct Loewner matrices in such a setting, which makes the newly introduced representation important for practitioners. Finally, the new formulation for $\L_2$ will be crucial for our proposed algorithm as the structure will be exploited in our computational framework. We note that the structure for $\L_2$ which we introduced in \eqref{eq:L2} appears in a similar form in \cite{ionita_data-driven_2014}. The main difference is that in \cite{ionita_data-driven_2014} the matrix $\L_2$ is only based on errors over the subset $\big(\bZ^{(1)} \setminus \blambda^{(1)}\big) \times \big(\bZ^{(2)} \setminus \blambda^{(2)}\big)$ rather than the full set of sample data. That representation only allows for defining $\L_2$ based on standard Cauchy matrices rather than the modified ones we introduced in \eqref{eq:modifiedCauchy}. However, we would like to consider all errors in $\bZ^{(1)} \times \bZ^{(2)}$ in the p-AAA LS problem. Therefore, we utilize our newly introduced form for $\L_2$ which generalizes the form introduced in \cite{ionita_data-driven_2014}.  
	
	With the LS error expression in hand we consider the p-AAA LS problem
	\begin{equation}
		\label{eq:paaals}
		\min_{\fronorm{\alpha}=1} \lVert \L_2 \vectorize(\alpha) \rVert_2^2.
	\end{equation}
	Combining the LS problem above with the greedy selection outlined in \eqref{eq:greedy} yields the p-AAA algorithm which we summarize in Algorithm~\ref{alg:paaa}.
	\begin{algorithm}[t]
		\caption{p-AAA}\label{alg:paaa}
		\begin{algorithmic}[1]
			\Require{Sets of sampling points $\bZ^{(1)}$, $\bZ^{(2)}$, and data matrix $\D$}
			\Ensure{$\br \approx \bff$}
			\State{Initialize $\br \equiv \operatorname{average}(\D)$}
            \State{$\blambda^{(1)},\blambda^{(2)} \gets \emptyset$}
			\While{error $>$ desired tolerance}
			\State{Select $ \left( \lambda^{(1)}_*, \lambda^{(2)}_* \right) $ via a greedy selection}
			\State Update the interpolation sets:
			\State\hspace{\algorithmicindent} $\blambda^{(1)} \gets \blambda^{(1)} \cup \{ \lambda^{(1)}_* \}$
			\State\hspace{\algorithmicindent} $\blambda^{(2)} \gets \blambda^{(2)} \cup \{ \lambda^{(2)}_* \}$
			\State{Solve $ \min \lVert \L_2 \vectorize(\alpha) \rVert_2 $ s.t. $\fronorm{\alpha}=1$}
			\State{Use $\alpha$ to update the rational approximant $\br$}
			\State{error $ \gets \max_{i_1,i_2} \; \left\lvert \D_{i_1 i_2} - \br\left(Z^{(1)}_{i_1},Z^{(2)}_{i_2}\right) \right\vert / \left\lvert \D_{i_1 i_2} \right\rvert $}
			\EndWhile
		\end{algorithmic}
	\end{algorithm}
	We note that the constraint $\fronorm{\alpha} = \lVert \vectorize(\alpha) \rVert_2 = 1$ is present in order to avoid zero coefficients and the trivial solution $\alpha = 0$. Further, the constrained LS problem has a closed-form solution given by the right singular vector of $\L_2$ which relates to the smallest singular value of $\L_2$. By using standard SVD algorithms the number of required floating point operations to solve \eqref{eq:paaals} is in $\cO(N_1 N_2 n_1^2 n_2^2)$. \\
    
    We note that if $\L_2$ has a nullspace the rational function $\br$ becomes the so-called Loewner interpolant \cite{ionita_data-driven_2014}, which exactly recovers the full set of sampling data. Importantly, if $\bff$ is a rational function and sufficient data is available, p-AAA will either compute a rational approximant with a desired error tolerance or recover $\bff$ exactly. In the parametric Loewner framework the order $(n_1-1,n_2-1)$ that allows for recovering the entire sample data set is determined a priori via computing the ranks of several $1$-$D$ Loewner matrices. A key computational difference between the p-AAA and Loewner frameworks is therefore that the Loewner framework computes the ranks of many small and the nullspace of one large Loewner matrix, while p-AAA solves several linear LS problems involving Loewner matrices with increasing dimensions. We will not go into detail about further differences and variants of the Loewner approach but briefly comment on the recently proposed extension of the parametric Loewner framework from \cite{antoulas_loewner_2024}. The goal of the extension is the same as our goal in this paper: Adjust an established multivariate rational approximation algorithm such that it can be applied to problems that involve many variables. The main idea in \cite{antoulas_loewner_2024} is to exploit the fact that the nullspace of $\L_2$ has a recursive structure which can be used in an efficient algorithm for computing the barycentric coefficient matrix $\alpha$. In the p-AAA setting however, $\L_2$ will usually not have a nullspace and it is unclear how to incorporate the results established in \cite{antoulas_loewner_2024} into Algorithm~\ref{alg:paaa}. As will become clear in Sections~\ref{sec:lrbarycentric} and \ref{sec:lrpaaa} we use a completely different approach to tackle the curse of dimensionality. 
	
	\subsection{p-AAA for Many Variables}
	\label{sec:paaamanyparameters}
	Our goal for this section is to extend the p-AAA discussion, presented for the two variable-case above, to the general multivariate approximation setting with $d\in\N$ variables given by $\bz = \left( z^{(1)},\ldots,z^{(d)}\right) \in \C^d$. In this setting many of the underlying matrix structures are replaced by higher-order tensors. We begin by considering the sampling points
	\begin{equation}
		\label{eq:dsamplingpoints}
		\bZ^{(j)} = \left\{ Z^{(j)}_1,\ldots, Z^{(j)}_{N_j} \right\} \subset \C \quad \text{for} \quad j=1,\ldots,d
	\end{equation}
	and the sample tensor $\D \in \C^{N_1 \times \cdots \times N_d}$ given by
	\begin{equation}
		\label{eq:sampletensor}
		\D_{i_1\ldots i_d} = \bff\left(Z^{(1)}_{i_1},\ldots,Z^{(d)}_{i_d}\right) \quad \text{for} \quad i_j = 1,\ldots,N_j \quad \text{and} \quad j = 1,\ldots,d.
	\end{equation}
	By introducing the sets of interpolation nodes
	\begin{equation}
		\label{eq:dinterpolationnodes}
		\blambda^{(j)} = \left\{ \lambda^{(j)}_1,\ldots,\lambda^{(j)}_{n_j} \right\} \subset \bZ^{(j)} \quad \text{for} \quad j=1,\ldots,d
	\end{equation}
    we can define the tensor of interpolated data $\H \in \C^{n_1 \times \cdots \times n_d}$ whose entries are given by
    \begin{equation*}
        \H_{i_1\ldots i_d} = \bff\left(\lambda^{(1)}_{i_1},\ldots,\lambda^{(d)}_{i_d}\right) \quad \text{for} \quad i_j = 1,\ldots,n_j \quad \text{and} \quad j = 1,\ldots,d.
    \end{equation*}
	The newly established barycentric form~\eqref{eq:barycentricform} now reads as
	\begin{equation}
		\label{eq:manyparambarycentric}
		\br(\bz) = \left[ \cC^{(1)}\left(z^{(1)}\right) \otimes \cdots \otimes \cC^{(d)}\left(z^{(d)}\right)\right]^\top \vectorize(\alpha \circ \H) \Big/ \left[\cC^{(1)}\left(z^{(1)}\right) \otimes \cdots \otimes \cC^{(d)}\left(z^{(d)}\right) \right]^\top \vectorize(\alpha),
	\end{equation}
	where the barycentric coefficients are given in terms of the tensor $\alpha \in \C^{n_1 \times \cdots \times n_d}$. The main adjustment in the p-AAA algorithm is that the LS problem now reads
	\begin{equation}
		\label{eq:manyvariablepaaals}
		\min_{\fronorm{\alpha} = 1}\lVert \L_d \vectorize(\alpha) \rVert_2^2,
	\end{equation}
	where $\L_d \in \C^{N_1 \cdots N_d \times n_1 \cdots n_d}$ is the $d$-dimensional Loewner matrix given by
	\begin{equation}
		\label{eq:Ld}
		\L_d = \diag(\vectorize(\D)) \left[\cC^{(1)} \otimes \cdots \otimes \cC^{(d)}\right]^\top -  \left[\cC^{(1)} \otimes \cdots \otimes \cC^{(d)}\right]^\top \diag(\vectorize(\H)).
	\end{equation}
    Note that for a tensor $\fronorm{\cdot}$ simply denotes the standard norm for tensors given by $\fronorm{\alpha} = \lVert \vectorize(\alpha) \rVert_2$. Solving the LS problem that involves $\L_d$ via a standard SVD algorithm requires $\cO(N_1\cdots N_d n_1^2\cdots n_d^2)$ operations. Especially for problems that either require a high number of interpolation points $n_1 \cdots n_d$ or depend on many variables, solving the underlying LS problem will be computationally prohibitive.  A main contribution of this manuscript is the introduction of a new representation for the multivariate barycentric form, utilizing separable functions, as detailed in Section~\ref{sec:lrbarycentric}. This representation will leverage a low-rank decomposition of the tensor $\alpha$. Then, we present our main contribution by showing how the new low-rank barycentric forms can be incorporated into the p-AAA algorithm, as discussed in Section~\ref{sec:lrpaaa}. This will significantly reduce the computational complexity of the underlying LS problem and allow for applying the p-AAA framework in situations where a greater number of variables is present.
	
	\section{Separable Functions and Barycentric Forms}
	\label{sec:lrbarycentric}
	In this section we discuss representations of barycentric forms based on separable functions and their connection to low-rank tensor decompositions. These representations will serve as the foundation for our primary contribution, the low-rank p-AAA algorithm in
    Section~\ref{sec:lrpaaa}.
	
	\subsection{Low-Rank Barycentric Coefficients}
	\label{sec:separablerep}
	We begin with a canonical (CP) decomposition \cite{kolda_tensor_2009} for the tensor $\alpha \in \C^{n_1 \times \cdots \times n_d}$ given by
	\begin{equation}
		\label{eq:cpalpha}
		\vectorize(\alpha) =  \sum_{k=1}^r \beta^{(1)}_{k} \otimes \cdots \otimes \beta^{(d)}_{k}.
	\end{equation}
	In this representation~\eqref{eq:cpalpha} we call the matrices $\beta^{(j)} = \left[\beta^{(j)}_1,\ldots,\beta^{(j)}_r \right] \in \C^{n_j \times r}$ for $j=1,\ldots,d$ CP factors and $r$ the size of the CP decomposition. Additionally, the rank of a tensor is defined as the minimal size of a CP decomposition that exactly represents that tensor. We note that storing $\alpha$ requires $n_1 \cdots n_d$ values, whereas storing the CP factors involves $r(n_1 + \cdots + n_d)$ values. Hence, if $r$ is small compared to $n_1,\ldots,n_d$ storing the CP factors rather than the full tensor is significantly more memory efficient. In this case, we call the CP decomposition low-rank. For now, we will assume that some CP decomposition for $\alpha$ as above is given and only reveal in Section~\ref{sec:lrpaaa} how the CP factors are computed in the context of a low-rank version of the p-AAA algorithm. Before we analyze the implications of introducing low-rankness to the barycentric form, we first note a simple connection between the CP decomposition and low-rank matrix decompositions. Specifically, for the two variable case we have $\alpha \in \C^{n_1 \times n_2}$ and
	\begin{equation*}
		\vectorize(\alpha) = \sum_{k=1}^r \beta^{(1)}_{k} \otimes \beta^{(2)}_{k} \qquad \iff \qquad \alpha = \beta^{(1)} \\ {\beta^{(2)}}^\top.
	\end{equation*}
	Hence, the CP factors simply become factors in an outer product representation for $\alpha$. Next, we establish how this connects to separable functions in barycentric forms. Let us consider the denominator of the barycentric form in \eqref{eq:manyparambarycentric} in conjunction with the CP decomposition in \eqref{eq:cpalpha} for $\alpha$:
	\begin{align*}
		\bd(\bz) &= \left[\cC^{(1)}\left(z^{(1)}\right) \otimes \cdots \otimes \cC^{(d)}\left(z^{(d)}\right)\right]^\top \vectorize(\alpha) \\ 
		&= \sum_{k=1}^r \left[\cC^{(1)}\left(z^{(1)}\right) \otimes \cdots \otimes \cC^{(d)}\left(z^{(d)}\right)\right]^\top \left(\beta^{(1)}_{k} \otimes \cdots \otimes \beta^{(d)}_{k}\right) \\
		&= \sum_{k=1}^r \left(\cC^{(1)}\left(z^{(1)}\right)^\top \beta^{(1)}_{k} \right) \cdots \left(\cC^{(d)}\left(z^{(d)}\right)^\top \beta^{(1)}_{k} \right).
	\end{align*}
	We note that in the above representation $\bd(\bz)$ is now formed by a sum of $r$ products of $d$ univariate scalar functions $\cC^{(j)}\left(z^{(j)}\right)^\top \beta^{(j)}_{k} \in \C$. We call multivariate functions that can be represented by a product of univariate functions in the individual function arguments separable. Further, we call $r$ the separation rank of $\bd$ if the number of separable summands that represent $\bd$ is minimal. The benefit of this representation is that it only requires $\cO(dr(n_1+\cdots+n_d))$ multiplications to evaluate it in contrast to the $\cO(n_1\cdots n_d)$ multiplications required for the standard barycentric form. We note that the barycentric coefficients $\alpha$ are also the coefficients for the Lagrange basis functions that form the polynomial $\bQ(\bz)$ in \eqref{eq:largrangedenom}. In particular, we can perform the same derivation as we did for $\bd$ above and write $\br = \bP / \bQ$ with
	\begin{equation}
		\label{eq:separableQ}
		\bQ(\bz) = \sum_{k=1}^r \left(L^{(1)}\left(z^{(1)}\right)^\top \beta^{(1)}_{k} \right) \cdots \left(L^{(d)}\left(z^{(d)}\right)^\top \beta^{(1)}_{k} \right)
	\end{equation}
	where
	\begin{equation*}
		L^{(j)}\left(z^{(j)}\right) = \left[ \prod_{i_j \neq 1} \left( z^{(j)} - \lambda_{i_j}^{(j)}\right),\ldots, \prod_{i_j \neq d} \left( z^{(j)} - \lambda_{i_j}^{(j)}\right) \right]^\top \quad \text{for} \quad j=1,\ldots,d.
	\end{equation*}
	This representation reveals that low-rank barycentric coefficients allow for representing rational functions whose denominator is a sum of a small number of separable polynomials. \\
	
	Next, we consider the numerator of the barycentric representation in \eqref{eq:manyparambarycentric}. Here the situation is slightly different, as we need to consider the element-wise product $\vectorize(\alpha \circ \H)$ instead of just $\vectorize(\alpha)$. To obtain a separable representation for the numerator, we first assume that the tensor of interpolated samples $\H \in \C^{n_1 \times \cdots \times n_d}$ admits a CP decomposition of the form
	\begin{equation}
		\label{eq:cpH}
		\vectorize(\H) = \sum_{\ell=1}^{q} \gamma^{(1)}_{\ell} \otimes \cdots \otimes \gamma^{(d)}_{\ell}.
	\end{equation}
	Then, we can write the numerator as
	\begin{align*}
		\bn(\bz)  &= \left[\cC^{(1)}\left(z^{(1)}\right) \otimes \cdots \otimes \cC^{(d)}\left(z^{(d)}\right)\right]^\top \vectorize(\alpha \circ \H) \\ 
		&= \sum_{k=1}^r \sum_{\ell=1}^q\left[\cC^{(1)}\left(z^{(1)}\right) \otimes \cdots \otimes \cC^{(d)}\left(z^{(d)}\right)\right]^\top \left[ \left(\beta^{(1)}_{k} \otimes \cdots \otimes \beta^{(d)}_{k}\right) \circ \left(\gamma^{(1)}_{\ell} \otimes \cdots \otimes \gamma^{(d)}_{\ell}\right) \right]\\
		&= \sum_{k=1}^r \sum_{\ell=1}^q \left(\cC^{(1)}\left(z^{(1)}\right)^\top \left[\beta^{(1)}_{k} \circ \gamma^{(1)}_{\ell} \right]\right) \cdots \left(\cC^{(d)}\left(z^{(d)}\right)^\top \left[\beta^{(d)}_{k} \circ \gamma^{(d)}_{\ell} \right] \right).
	\end{align*}
	The representation above shows that $\bn$ has a separation rank which is bounded from above by $rq$. In practice, $\H$ will typically not be exactly low-rank and conditions under which $\H$ can be well-approximated by a low-rank tensor would need to be investigated carefully. As will become clear in Section~\ref{sec:lrpaaa} the rank of $\H$ will not be a concern for us. In particular, we achieve a significantly improved computational complexity in our proposed algorithm by focusing solely on low-rank representations of $\alpha$.

	\subsection{Motivation for Low-Rank Structures}
	The main motivation for using low-rank barycentric coefficients in the context of p-AAA is that they can be computed efficiently, as discussed in the next section. It is worth pointing out that the CP and several other types of tensor decomposition have long been the cornerstone of multivariate approximation algorithms \cite{grasedyck_literature_2013} and considering them in the context of p-AAA is therefore a natural next step. In order to provide additional motivation, we consider here two classes of problems for which separable and therefore also low-rank structures appear naturally. While the discussion in this section is considering exact separable structures, we note that we typically use representations based on separable functions as an approximation rather than exact representations. The model structure introduced in Lemma~\ref{lemma:blockseparablef} is particularly relevant for the parametric stationary PDE setting introduced in Section~\ref{sec:motivatingproblems}. First, we consider a key fact that applies to the approximation of an inherently separable function.
	\begin{lemma}
		\label{lemma:rank1data}
     For $j=1,\ldots,d$, consider the sampling points $\bZ^{(j)}$ in \eqref{eq:dsamplingpoints} and the interpolation nodes $\blambda^{(j)}$ in \eqref{eq:dinterpolationnodes}. Let $\D$ be the corresponding sample tensor in \eqref{eq:sampletensor} and $\L_d$ be the associated higher-order Loewner matrix in \eqref{eq:Ld}. If $\rk(\D) = 1$ and there exists $\alpha \neq 0$ such that $\L_d \vectorize(\alpha) = 0$, then we can choose $\alpha$ such that $\rk(\alpha) = 1$.
	\end{lemma}%
	{The proof of~\Cref{lemma:rank1data} can be found in Appendix~\ref{appendix:proofrank1lemma}. The lemma states that rank-$1$ data that can be exactly recovered by a rational function can, in particular, be exactly recovered by a separable rational function. Note that the data matrix $\D$ has rank-$1$ if the function $\bff$ is separable to begin with. This will be the case,  for example, if the parameter dependence of an LTI model as in \eqref{eq:lti} with $\bp = \left(p^{(1)},p^{(2)}\right) \in \C^2$ is restricted to the inputs and outputs via $\bb\left(p^{(1)}\right)$ and $\bc\left(p^{(2)}\right)$. For this example the transfer function reads
    \begin{equation*}
        \bH\left(s,p^{(1)},p^{(2)}\right) = \bc\left(p^{(2)}\right)^\top \left(s\bE - \bA\right)^{-1} \bb\left(p^{(1)}\right)
    \end{equation*}
    and is clearly separable.} For such functions it is wise to exploit the structure for example by enforcing the barycentric coefficients of a rational approximant to be rank-$1$. Next, we introduce another class of functions which exhibit a separable structure.
	\begin{lemma}
		\label{lemma:blockseparablef}
		Let $\bz = \left(z^{(1)},z^{(2)}\right) \in \C^2$ and
		\begin{equation*}
			\bff(\bz) = \bc^\top \bK(\bz)^{-1} \bb
		\end{equation*}
		with
		\begin{equation}
			\label{eq:blockK}
			\bK(\bz) = \begin{bmatrix}
				z^{(1)} \bK_{11} & \bK_{12} \\
				\bK_{21} & z^{(2)} \bK_{22}
			\end{bmatrix} \in \C^{(m_1 + m_2) \times (m_1 + m_2)},
		\end{equation}
		where $\bK_{11} \in \C^{m_1 \times m_1}$ and $\bK_{22} \in \C^{m_2 \times m_2}$ are invertible matrices. Further, let $r = \min\{\rk(\bK_{12}),\allowbreak \rk(\bK_{21})\}$ be the minimum of the ranks of the off-diagonal blocks $\bK_{12},\bK_{21}^\top \in \C^{m_1 \times m_2}$. Then $\bff$ is a rational function 
		\begin{equation*}
			\bff(\bz) = \frac{p(\bz)}{q(\bz)},
		\end{equation*}
		where $p,q$ are polynomials and $q$ can be written as
		\begin{equation*}
			q(\bz) = \sum_{\ell=1}^{r+1} q^{(1)}_{\ell}\left(z^{(1)}\right) q^{(2)}_{\ell}\left(z^{(2)}\right) 
		\end{equation*}
		with polynomials $q^{(1)}_{\ell}, q^{(2)}_{\ell}$ for $\ell=1,\ldots,r+1$. Hence, the separation rank of the denominator $q$ is less than or equal to $r+1$. 
	\end{lemma}
	The proof of Lemma~\ref{lemma:blockseparablef} can be found in Appendix~\ref{appendix:proofblockKlemma}. The lemma establishes a class of rational functions that exhibit a separable structure in their denominator polynomial and thus in their barycentric coefficients if represented in barycentric form. In particular, the separation rank is bounded by the rank of the off-diagonal blocks $r$ of the matrix $\bK(\bz)$ and is independent of the order of the rational function given by $(m_1, m_2)$. Our result can easily be extended to $d>2$ variables, where
	\begin{equation*}
		\bK(\bz) = \begin{bmatrix}
			z^{(1)} \bK_{11} & \cdots & \bK_{1d} \\
			\vdots & \ddots & \vdots \\
			\bK_{d1} & \cdots & z^{(d)} \bK_{dd}
		\end{bmatrix}.
	\end{equation*}
	In this case the separation rank of the denominator of $\bff$ will be bounded by $(r+1)^{d(d-1)/2}$ where $r$ is the minimum rank of all off-diagonal blocks of $\bK(\bz)$ whose rank is non-zero. {This bound can in many cases be improved by taking into account additional structure in the matrix $\bK(\bz)$. For example, if $\bK(\bz)$ is block tridiagonal we can bound the separation rank by $(r+1)^{d-1}$.} Most importantly, we note that this class of rational functions commonly appears in the context of spatially discretized parametric PDEs. In particular, if the effects of the parameters (here $z^{(1)},\ldots,z^{(d)}$) are in a sense local and independent, $\bK(\bz)$ will have the exact structure introduced above. An example for such a model is considered in our numerical experiment discussed in Section~\ref{sec:thermalblock}.
    
	\section{The Low-Rank p-AAA Algorithm}
	\label{sec:lrpaaa}
	In this section we present our main contribution in the form of the low-rank p-AAA algorithm. As before, we begin by primarily focusing on the  $d=2$ case in order to make the main ideas as clear as possible. We extend the result to the case $d>2$ afterwards and discuss various practical aspects of the proposed algorithm.
	%\subsection{Low-Rank p-AAA for Two Variables}
    
	The first goal is to incorporate the previously introduced CP decomposition from Section~\ref{sec:lrbarycentric} for the barycentric coefficients into the p-AAA algorithm. We propose to achieve this goal by making the structure in \eqref{eq:cpalpha} a constraint for the LS optimization problem in \eqref{eq:paaals} appearing in p-AAA. The proposed algorithm, which we call low-rank p-AAA, therefore considers the constrained optimization problem
	\begin{equation*}
		\min_{\fronorm{\alpha} =1} \lVert \L_2 \vectorize(\alpha) \rVert_2^2 \quad \text{s.t.} \quad \vectorize(\alpha) = \sum_{k=1}^r \beta^{(1)}_{k} \otimes \beta^{(2)}_{k}
	\end{equation*} 
    in place of~\eqref{eq:paaals} in p-AAA. The greedy selection strategy~\eqref{eq:greedy} proceeds as before. 
    
	Plugging the low-rank constraint for $\vectorize(\alpha)$ into the LS objective function yields the optimization problem
	\begin{equation}
		\label{eq:twovariablelrpaaals}
		\min_{\beta^{(1)},\beta^{(2)}} \left\lVert \L_2 \sum_{k=1}^r \beta^{(1)}_{k} \otimes \beta^{(2)}_{k} \right\rVert_2^2 \quad \text{s.t.} \quad \left\lVert \sum_{k=1}^r \beta^{(1)}_{k} \otimes \beta^{(2)}_{k} \right\rVert_2 = 1.
	\end{equation}
	First, we define the matrices \textcolor{black}{
	\begin{equation}
		\label{eq:Ktwovariables}
		\begin{aligned}
			\bJ^{(1)} &= \left[ \bI_{n_1} \otimes \beta^{(2)}_1,\ldots, \bI_{n_1} \otimes \beta^{(2)}_r \right] \in \C^{n_1 n_2 \times n_1 r},~\mbox{and} \\
			\bJ^{(2)} &= \left[ \beta^{(1)}_1 \otimes \bI_{n_2}, \ldots, \beta^{(1)}_r \otimes \bI_{n_2} \right] \in \C^{n_1 n_2 \times n_2 r}.      
		\end{aligned}
	\end{equation}}
	Based on the relation
	\begin{equation*}
		\beta^{(1)}_{k} \otimes \beta^{(2)}_{k} = \left(\beta^{(1)}_{k} \otimes \bI_{n_2}\right) \beta^{(2)}_{k} = \left(\bI_{n_1} \otimes \beta^{(2)}_{k}\right) \beta_{k}^{(1)},
	\end{equation*}
	we obtain
	\begin{equation} \label{eq:alphaK1K2}
		\vectorize(\alpha) = \sum_{k=1}^r \beta^{(1)}_{k} \otimes \beta^{(2)}_{k} = \bJ^{(1)} \vectorize\left(\beta^{(1)}\right) = \bJ^{(2)} \vectorize\left(\beta^{(2)}\right).
	\end{equation}
	Next, we define the \emph{contracted} Loewner matrices
	\begin{equation} \label{eq:Loewcontracted}
		\L^{(1)}_2 := \L_2 \bJ^{(1)} \in \C^{N_1 N_2 \times n_1 r} \qquad \text{and} \qquad \L^{(2)}_2 := \L_2 \bJ^{(2)} \in \C^{N_1 N_2 \times n_2 r}.
	\end{equation}
Then, employing~\eqref{eq:alphaK1K2} and~\eqref{eq:Loewcontracted}  allows for equivalently rewriting the objective function as
	\begin{equation} \label{eq:objfuncnew}
		\left\lVert \L_2 \sum_{k=1}^r \beta^{(1)}_{k} \otimes \beta^{(2)}_{k} \right\rVert_2^2 = \left\lVert \L^{(1)}_2 \vectorize\left(\beta^{(1)}\right) \right\rVert_2^2 = \left\lVert \L^{(2)}_2 \vectorize\left(\beta^{(2)}\right) \right\rVert_2^2.
	\end{equation}
	Overall this reformulation~\eqref{eq:objfuncnew} yields two equivalent ways for writing the LS problem in \eqref{eq:twovariablelrpaaals}: 
	\begin{align}
	\mbox{(i)}~~~~	\min_{\beta^{(1)}, \beta^{(2)}} \left\Vert \L_2^{(1)} \vectorize\left(\beta^{(1)}\right) \right\rVert_2^2 \quad \text{s.t.} \quad \left\lVert \bJ^{(1)} \vectorize\left(\beta^{(1)}\right) \right\rVert_2 = 1 , \label{eq:beta1ls} \\
	\mbox{(ii)}~~~~	\min_{\beta^{(1)}, \beta^{(2)}} \left\Vert \L_2^{(2)} \vectorize\left(\beta^{(2)}\right) \right\rVert_2^2 \quad \text{s.t.} \quad \left\lVert \bJ^{(2)} \vectorize\left(\beta^{(2)}\right) \right\rVert_2 = 1. \label{eq:beta2ls}
	\end{align}
    Recall the definitions of $\bJ^{(1)}$
    and $\L_2^{(1)}$ from~\eqref{eq:Ktwovariables} and~\eqref{eq:Loewcontracted}.  
	An important observation is now that
   if  $\beta^{(2)}$ is fixed in~\eqref{eq:beta1ls}, then $\beta^{(1)}$
    can be computed by solving a constrained linear LS problem which admits a closed form solution in terms of a generalized SVD \cite{van_loan_generalizing_1976}. Similarly, if  $\beta^{(1)}$ is fixed in~\eqref{eq:beta2ls}, then $\beta^{(2)}$
    can be computed. Suppose, for example, that $\beta^{(2)}$ is fixed and we would like to solve the optimization problem in \eqref{eq:beta1ls} for the free variable $\beta^{(1)}$. Then we can compute the (economy size) generalized SVD of the matrix tuple $\big(\L_2^{(1)},\bJ^{(1)}\big)$ given by
	\begin{align*}
		\bU^\top \L_2^{(1)} \bX &= \bC, \\
		\bV^\top \bJ^{(1)} \bX &= \bD,
	\end{align*}
	where $\bU\in\C^{N_1 N_2 \times n_1 r}$ and $\bV\in\C^{n_1 n_2 \times n_1 r}$ are sub-unitary, $\bC,\bD \in \C^{n_1 r \times n_1 r}$ are diagonal matrices with non-negative entries, and $\bX \in \C^{n_1 r \times n_1 r}$ is an invertible matrix. Next, assuming that $\bJ^{(1)}$ is not rank-deficient, we get that the generalized singular values are given by $\sigma_i = \bC_{ii} / \bD_{ii}$. Then the solution $\beta^{(1)}$ of the constrained LS problem~\eqref{eq:beta1ls} is finally given by
	\begin{equation}
		\label{eq:gsvdsolution}
		\vectorize\left(\beta^{(1)}\right) = \frac{1}{\bD_{kk}} \bX e_k,
	\end{equation}
	where $e_k \in \R^{n_1 n_2}$ is the $k$-th unit vector and $k$ is the index of the smallest generalized singular value $\sigma_k$. Similarly, we can compute the solution of the LS problem in \eqref{eq:beta2ls} if $\beta^{(1)}$ is fixed by computing a generalized SVD of the tuple $\big(\L_2^{(2)},\bJ^{(2)}\big)$. It is important to note that the generalized SVDs of the tuples $\big(\L_2^{(1)},\bJ^{(1)}\big)$ and $\big(\L_2^{(2)},\bJ^{(2)}\big)$ can be computed in $\cO(N_1 N_2 n_1^2 r^2)$ and $\cO(N_1 N_2 n_2^2 r^2)$, respectively. Hence, if $r \ll n_1,n_2$ then computing the generalized SVDs is significantly faster than computing the SVD of $\L_2$. This fact motivates using an alternating LS (ALS) solver, based on the two formulations depicted in \eqref{eq:beta1ls} and \eqref{eq:beta2ls} to tackle the optimization problem in \eqref{eq:twovariablelrpaaals}. 

 \begin{algorithm}[t]
		\caption{ALS for low-rank p-AAA}\label{alg:als}
		\begin{algorithmic}[1]
			\Require{Data $\bZ^{(j)},\blambda^{(j)},\D,\H$, initial guess for $\beta^{(j)} \in \C^{n_j \times r}$ for $j=1,\ldots,d$}
			\Ensure{$\beta^{(1)},\ldots,\beta^{(d)}$}
			\Repeat
            \For{$j=1,\ldots,d$}
			\State $\beta^{(j)} \gets \argmin_{\beta^{(j)}} \left\lVert \L_d^{(j)} \vectorize\left(\beta^{(j)}\right) \right\rVert_2^2$ \text{ s.t. } $\left\lVert \bJ^{(j)} \vectorize\left(\beta^{(j)}\right) \right\rVert_2 = 1$
			\State Store column norms $\mu \gets \left[ \lVert \beta^{(j)}_1 \rVert_2,\ldots, \lVert \beta^{(j)}_r \rVert_2 \right]$
			\State Normalize columns $\beta^{(j)} \gets \beta^{(j)} \diag(\mu)^{-1}$
			\EndFor
			\Until{$\lVert \L_2 \vectorize(\alpha) \rVert_2^2$ ceases to change}
			\State $\beta^{(1)} \gets \beta^{(1)} \diag(\mu)$
		\end{algorithmic}
	\end{algorithm}

    In order to extend this derivation to $d>2$ variables we consider
	\begin{equation*}
		\vectorize(\alpha) = \sum_{k=1}^r \beta^{(1)}_{k} \otimes \cdots \otimes \beta^{(d)}_{k} \in \C^{n_1 \cdots n_d}
	\end{equation*}
	and the corresponding constrained optimization problem
	\begin{equation}
		\label{eq:lrpaaals}
		\min_{\beta^{(1)},\ldots,\beta^{(d)}} \left\lVert \L_d \sum_{k=1}^r \beta^{(1)}_{k} \otimes \cdots \otimes \beta^{(d)}_{k} \right\rVert_2^2 \quad \text{s.t.} \quad \left\lVert \sum_{k=1}^r \beta^{(1)}_{k} \otimes \cdots \otimes \beta^{(d)}_{k} \right\rVert_2 = 1.
	\end{equation}
	The generalized form of the matrices $\bJ^{(1)}, \bJ^{(2)}$ in \eqref{eq:Ktwovariables} now reads
	\begin{equation}
		\label{eq:Kmanyvariables}
		\begin{aligned}
			\bJ^{(j)} = [\beta^{(1)}_1 \otimes \cdots \otimes \beta^{(j-1)}_1 &\otimes \bI_{n_j} \otimes \beta^{(j+1)}_1 \cdots \otimes \beta^{(d)}_1, \ldots, \\ & \beta^{(1)}_r \otimes \cdots \otimes \beta^{(j-1)}_r \otimes \bI_{n_j} \otimes \beta^{(j+1)}_r \cdots \otimes \beta^{(d)}_r] \in \C^{n_1 \cdots n_d \times n_j r},
		\end{aligned}
	\end{equation}
	for $j=1,\ldots,d$. Similar to the LS problems in \eqref{eq:beta1ls} and \eqref{eq:beta2ls} we now obtain
	\begin{equation} \label{eq:Ldjmin}
		\min_{\beta^{(1)},\ldots,\beta^{(d)}} \left\Vert \L_d^{(j)} \vectorize\left(\beta^{(j)}\right) \right\rVert_2^2 \quad \text{s.t.} \quad \left\lVert \bJ^{(j)} \vectorize\left(\beta^{(j)}\right) \right\rVert_2 = 1,
	\end{equation}
	where $\L_d^{(j)} = \L_d \bJ^{(j)}$ for $j=1,\ldots,d$. The situation is similar to the $d=2$ case. For example, 
    when  $\beta^{(2)},\beta^{(3)},\ldots,\beta^{(d)}$ are fixed in~\eqref{eq:Ldjmin}, then $\beta^{(1)}$
    can be directly computed by solving the resulting constrained LS problem.     We summarize the corresponding ALS approach in Algorithm~\ref{alg:als}. Note that we rescale the columns of the CP factors in order to avoid numerical issues such as (nearly) singular factor matrices and ill-conditioned LS matrices along the lines of ALS algorithms which are used for computing a CP decomposition based on a given full tensor \cite{kolda_tensor_2009}.
    
	Now that we have introduced the ALS algorithm in the setting of p-AAA, we can also formulate our proposed low-rank p-AAA algorithm, which is depicted in Algorithm~\ref{alg:lrpaaa}. The two key differences between the original p-AAA as  in Algorithm~\ref{alg:paaa} and the proposed low-rank version as in Algorithm~\ref{alg:lrpaaa} are that (i) in low-rank p-AAA the LS problem is solved via the ALS procedure outlined in Algorithm~\ref{alg:als} and (ii) the barycentric coefficients are stored implicitly in terms of the CP factors. 
	\begin{algorithm}[ht]
		\caption{Low-rank p-AAA}\label{alg:lrpaaa}
		\begin{algorithmic}[1]
			\Require{Sets of sampling points $\bZ^{(1)},\ldots,\bZ^{(d)}$, data tensor $\D$, size of CP decomposition $r$}
			\Ensure{$\br \approx \bff$}
			\State{Initialize $\br \equiv \operatorname{average}(\D)$}
            \State{$\blambda^{(1)},\ldots,\blambda^{(d)} \gets \emptyset$}
			\While{error $>$ desired tolerance}
			\State{Select $ \left( \lambda^{(1)}_*,\ldots,\lambda^{(d)}_* \right) $ via a greedy selection}
            \For{$j=1,\ldots,d$}
			\State Update the interpolation set $\blambda^{(j)} \gets \blambda^{(j)} \cup \{ \lambda^{(j)}_* \}$
			\EndFor
			\State{Use Algorithm~\ref{alg:als} to solve\textcolor{black}{
            $$
                \min_{\beta^{(1)},\ldots,\beta^{(d)}} \left\lVert \L_d \sum_{k=1}^r \beta^{(1)}_{k} \otimes \cdots \otimes \beta^{(d)}_{k} \right\rVert_2^2 \quad \text{s.t.} \quad \left\lVert \sum_{k=1}^r \beta^{(1)}_{k} \otimes \cdots \otimes \beta^{(d)}_{k} \right\rVert_2 = 1
            $$}}
			\State{Use $\beta^{(1)},\ldots,\beta^{(d)}$ to update the rational approximant $\br$}
			\State{error $ \gets \max_{i_1,\ldots,i_d} \; \left\lvert \D_{i_1 \ldots i_d} - \br\left(Z^{(1)}_{i_1},\ldots,Z^{(d)}_{i_d}\right) \right\vert / \left\lvert \D_{i_1 \ldots i_d} \right\rvert $}
			\EndWhile
		\end{algorithmic}
	\end{algorithm}
    
	There are two important computational advantages of using low-rank p-AAA over the standard p-AAA algorithm. First, the number of floating point operations required for the low-rank p-AAA algorithm in Algorithm~\ref{alg:lrpaaa} is on the order of $\cO(N_1 \cdots N_d r^2(n_1^2 + \cdots + n_d^2))$ while the standard p-AAA LS problem admits a computational complexity of $\cO(N_1 \cdots N_d n_1^2 \cdots n_d^2)$. Thus, if $r \ll \min(n_1,n_2,\ldots,n_d)$ ALS allows for computing the factors $\beta^{(j)}$ for $j=1,\ldots,d$ significantly more efficiently than the standard LS approach. Additionally, we can avoid forming $\L_d$ explicitly in the ALS procedure as we explain in Section~\ref{sec:contractedLconstruction}. This drastically reduces the memory required for executing Algorithm~\ref{alg:lrpaaa} rather than the standard p-AAA algorithm.
    The matrices $\L_d^{(j)}$ can be constructed carefully such that the memory requirement of the contracted Loewner matrices never exceeds $\cO(N_1 \cdots N_d n_j r)$ for $j=1,\ldots,d$. Thus, for large data sets memory limitations are less of an issue in the ALS procedure if again $r \ll \min(n_1, n_2,\ldots,n_d)$. Additionally, it is worth pointing out that the ALS objective function is guaranteed to be non-increasing. Hence, in each iteration it is guaranteed that we obtain an improved solution or that the algorithm has converged.

	\section{Practical Considerations}
	\label{sec:practical}
	The goal of this section is to discuss how the choice of separation rank $r$, initial guesses for $\beta^{(j)}$, and convergence criteria can be chosen in Algorithms~\ref{alg:als} and \ref{alg:lrpaaa}, respectively. For the simplicity, we will perform some of our discussion below for the two-variable case but  as before, our conclusions easily extend to $d>2$ variables by employing the  formulae from Section~\ref{sec:lrpaaa} corresponding to the $d>2$ case.
	
	\subsection{CP Decomposition Size for Barycentric Coefficients}
	\label{sec:cpsize}
	In general, it is not clear how to make a good choice for the input $r$ in Algorithm~\ref{alg:lrpaaa}. Choosing $r=1$ obviously yields the most significant computational speed up, but for many problems a rational approximant with higher order will be required for the algorithm to converge. On the other hand, choosing $r$ too large may negate any computational advantage that low-rank p-AAA has over the original algorithm. A good balance between approximation quality and speed up for our numerical experiments was achieved for values of $r$ between $2$ and $10$. An ideal choice will however be problem dependent and more rigorous strategies for choosing $r$ remain to be investigated. Additionally, we note that if $r$ is chosen too large, the CP factors $\beta^{(1)},\ldots,\beta^{(d)}$ may have linearly dependent columns. If $r>1$ this is even guaranteed to happen in the first iteration of p-AAA where $\alpha \in \C^{1 \times \cdots \times 1}$ is a scalar. If $\beta^{(1)},\ldots,\beta^{(d)}$ have linearly dependent columns then the matrices in $\eqref{eq:Ktwovariables}$ will have linearly dependent columns as well. In this case the solution in \eqref{eq:gsvdsolution} is not well-defined. In order to deal with this issue, we propose computing $\rk(\bJ^{(1)}) = n_1 \rk(\beta^{(1)})$ prior to the generalized SVD and reducing $r$ to $\rk(\beta^{(1)})$ if the matrix is rank-deficient. Doing this will ensure that the solution in \eqref{eq:gsvdsolution} is well-defined. We suggest to perform similar rank-truncation steps for the matrix $\bJ^{(2)}$ or the matrices $\bJ^{(j)}$ in \eqref{eq:Kmanyvariables} for $j=1,\ldots,d$ in the many variable case.
	
	\subsection{ALS Initialization}
	\label{sec:alsinit}
	Another choice one has to make in the low-rank p-AAA algorithm are the initial guesses for $\beta^{(j)}$ in the ALS procedure outlined in Algorithm~\ref{alg:als}. While some advanced initialization strategies for ALS exist in the context of tensor approximations \cite{kolda_tensor_2009}, random matrices are a commonly employed choice. We propose here an effective initialization scheme for the low-rank p-AAA algorithm, which significantly outperforms random initial guesses for most of our numerical experiments. The main idea is to reuse the barycentric coefficients from previous iterations. Let $\beta^{(j)}(\ell) \in \C^{n_j \times r}$ be the $j$-th CP factor of $\alpha$ in step $\ell$ of Algorithm~\ref{alg:lrpaaa}. Accordingly, $\beta^{(j)}(\ell+1) \in \C^{\tn_j \times r}$ is the $j$-th CP factor of $\alpha$ which we would like to compute in step $\ell+1$ of the low-rank p-AAA algorithm via ALS. We have $\tn_j = n_j + 1$ if a new interpolation node was added to $\blambda^{(j)}$ after the $\ell$-th iteration and $\tn_j = n_j$ otherwise. In the latter case, we propose to initialize ALS in step $\ell+1$ with the previously computed iterate $\beta^{(j)}(\ell)$. In other words we choose $\beta^{(j)}(\ell)$ as the ALS initial guess for computing $\beta^{(j)}(\ell+1)$. If a new interpolation node was added to $\blambda^{(j)}$, we instead add a row of zeros to the initial guess of the ALS procedure. This means that we use $\begin{bmatrix} \beta^{(j)}(\ell)^\top & 0 \end{bmatrix}^\top \in \C^{(n_j+1) \times r}$ as the ALS initial guess. We use this process to initialize CP factors for $j=1,\ldots,d$. This initialization procedure guarantees that the the objective function $\lVert \L_d \vectorize(\alpha) \rVert_2^2$ at the beginning of ALS at step $\ell+1$ of low-rank p-AAA will be less than or equal to the ALS objective function at the end of the $\ell$-th p-AAA iteration. This fact is not tied to the currently considered low-rank structure and a general statement which summarizes this observation is given in the following lemma.
	\begin{lemma}
		\label{proposition:lsinit}
		Consider the sampling points $\bZ^{(j)}$ in \eqref{eq:dsamplingpoints} and the interpolation nodes $\blambda^{(j)}$ in \eqref{eq:dinterpolationnodes} for $j=1,\ldots,d$. Let $\L_d$ be the corresponding higher-order Loewner matrix in \eqref{eq:Ld} and $\alpha \in \C^{n_1 \times \cdots \times n_d}$ an arbitrary tensor. Let $\tilde{\blambda}^{(j)} = \left\{ \tilde{\lambda}^{(j)}_1, \ldots, \tilde{\lambda}^{(j)}_{\tn_j}\right\} \subset \bZ^{(j)}$ be any set of interpolation nodes such that $\blambda^{(j)} \subseteq \tilde{\blambda}^{(j)}$ and $\tn_j \geq n_j$. Additionally, let $\tilde{\L}_d$ be the higher-order Loewner matrix based on this set of interpolation nodes and let the tensor $\tilde{\alpha} \in \C^{\tn_1 \times \cdots \times \tn_d}$ be defined as
		\begin{equation*}
			\tilde{\alpha}_{i_1\ldots i_d} = \begin{cases}
				\alpha_{i_1\ldots i_d}  \quad &\text{if} \quad i_j \leq n_j \; \text{for} \; j=1,\ldots,d \\
				0 \quad &\text{else}
			\end{cases}
		\end{equation*}
		In other words, let $\tilde{\alpha}$ be zero everywhere except for the leading $n_1 \times \cdots \times n_d$ tensor which equals $\alpha$. Then
		\begin{equation*}
			\lVert \tilde{\L}_d \vectorize(\tilde{\alpha}) \rVert_2 \leq \lVert \L_d \vectorize(\alpha) \rVert_2.
		\end{equation*}
	\end{lemma}
	\begin{proof}
		First, consider the expressions
		\begin{equation*}
			\bE_{j_1\ldots j_d}(\alpha) = \sum_{i_1=1}^{n_1} \ldots \sum_{i_d=1}^{n_d} \alpha_{i_1\ldots i_d}  \left[\bff\left(Z^{(1)}_{j_1}, \ldots, Z^{(d)}_{j_d}\right) - \bff\left(Z^{(1)}_{i_1}, \ldots, Z^{(d)}_{i_d}\right) \right] \prod_{k=1}^d \cC\left(\blambda^{(k)},Z^{(k)}_{j_k}\right)_{i_k}
		\end{equation*}
		and
		\begin{equation*}
			\tilde{\bE}_{j_1\ldots j_d}(\tilde{\alpha}) = \sum_{\ti_1=1}^{\tn_1} \ldots \sum_{\ti_d=1}^{\tn_d} \tilde{\alpha}_{\ti_1\ldots \ti_d}  \left[\bff\left(Z^{(1)}_{j_1}, \ldots, Z^{(d)}_{j_d}\right) - \bff\left(Z^{(1)}_{\ti_1}, \ldots, Z^{(d)}_{\ti_d}\right) \right] \prod_{k=1}^d \cC\left(\tilde{\blambda}^{(k)},Z^{(k)}_{j_k}\right)_{\ti_k}.
		\end{equation*}
		It is easily verified that
		\begin{equation*}
			\lVert \L_d \vectorize(\alpha) \rVert_2^2 = \sum_{j_1=1}^{N_1}\cdots \sum_{j_d=1}^{N_d} \bE_{j_1\ldots j_d}(\alpha)^2 \quad \text{and} \quad \lVert \tilde{\L}_d \vectorize(\tilde{\alpha}) \rVert_2^2 = \sum_{j_1=1}^{N_1}\cdots \sum_{j_d=1}^{N_d} \tilde{\bE}_{j_1\ldots j_d}(\tilde{\alpha})^2.
		\end{equation*}
		Further, we note that the definition for $\tilde{\alpha}$ yields
		\begin{equation*}
			\tilde{\bE}_{j_1\ldots j_d}(\tilde{\alpha}) = \begin{cases}
				\bE_{j_1\ldots j_d}(\alpha) \quad &\text{if} \quad \left(Z^{(1)}_{j_1},\ldots,Z^{(d)}_{j_d}\right) \notin \tilde{\blambda}^{(1)} \times \cdots \times \tilde{\blambda}^{(d)} \\
				0 \quad &\text{else}
			\end{cases}
		\end{equation*}
		Since $\blambda^{(1)} \times \cdots \times \blambda^{(d)} \subseteq \tilde{\blambda}^{(1)} \times \cdots \times \tilde{\blambda}^{(d)}$ we therefore have
		\begin{equation*}
			\tilde{\bE}_{j_1\ldots j_d}(\tilde{\alpha})^2 \leq \bE_{j_1\ldots j_d}(\alpha)^2 \quad \text{for} \quad j_k = 1,\ldots,N_k \quad \text{and} \quad k=1,\ldots,d.
		\end{equation*}
		With this result in hand, we finally obtain
		\begin{equation*}
			\lVert \tilde{\L}_d \vectorize(\tilde{\alpha}) \rVert_2^2 = \sum_{j_1=1}^{N_1}\cdots \sum_{j_d=1}^{N_d} \tilde{\bE}_{j_1\ldots j_d}(\tilde{\alpha})^2 \leq \sum_{j_1=1}^{N_1}\cdots \sum_{j_d=1}^{N_d} \bE_{j_1\ldots j_d}(\alpha)^2 = \lVert \L_d \vectorize(\alpha) \rVert_2^2.
		\end{equation*}
	\end{proof}
	In summary, Lemma~\ref{proposition:lsinit} states that after adding interpolation points and adding zero values to the previous barycentric coefficient tensor $\alpha$, we are guaranteed that the linear LS error does not increase. Note that our previously suggested ALS initialization for the CP factors yields $\tilde{\alpha}$ that satisfies the assumptions of the proposition via
	\begin{equation*}
		\tilde{\alpha} = \sum_{k=1}^r \tilde{\beta}^{(1)}_k \otimes \cdots \otimes \tilde{\beta}^{(d)}_k,
	\end{equation*}
    where
    \begin{equation*}
		\label{eq:alsinit}
        \tilde{\beta}^{(j)} = \begin{cases}
            \beta^{(j)}(\ell) \quad &\text{if} \quad \tilde{n}_j = n_j \\
            \begin{bmatrix}
			 \beta^{(j)}(\ell)^\top & 0
		  \end{bmatrix}^\top \quad &\text{else}.
        \end{cases}
    \end{equation*}
	The discussed ALS initialization strategy has the advantageous property that the low-rank p-AAA algorithm is guaranteed to monotonically decrease the linearized approximation error within an individual and in between two successive applications of Algorithm~\ref{alg:als}, assuming $r$ was not reduced in the previous iteration as outlined in Section~\ref{sec:cpsize}. This effect is further illustrated in Section~\ref{sec:trigexample3} where the convergence behavior of low-rank p-AAA is assessed via a numerical example. We note that if $r$ was reduced in a previous p-AAA iteration in order to avoid rank-deficiencies, we initialize columns which are newly added to the ALS initial guess $\tilde{\beta}^{(j)}$ randomly.

	\subsection{ALS Stopping Criterion}
	\label{sec:alsstopping}
	Another choice one has to make for the ALS procedure is when to stop the iteration. If we perform too many iterations in the ALS procedure, we may negate the computational advantage our proposed low-rank approach has over the standard p-AAA algorithm. As ALS is known to converge slowly in the context of tensor approximation algorithms \cite{kolda_tensor_2009}, this is a concern we should take into consideration. Our proposed stopping criterion for ALS is based on the relative change in the ALS objective function. In particular, we suggest choosing $0 < \varepsilon < 1$ and terminating ALS if 
	\begin{equation*}
		\left\lvert \lVert \L_2 \vectorize\left(\alpha(\ell)\right) \rVert_2^2 - \lVert \L_2 \vectorize\left(\alpha(\ell+1)\right) \rVert_2^2 \right\rvert \leq \varepsilon \lVert \L_2 \vectorize\left(\alpha(\ell+1)\right) \rVert_2^2
	\end{equation*}
	where $\alpha(\ell)$ and $\alpha(\ell+1)$ are the barycentric coefficient tensors in steps $\ell$ and $\ell+1$ of ALS, respectively. Additionally, the ALS procedure should always be terminated if $\lVert \L_2 \vectorize(\alpha(\ell+1)) \rVert_2^2 = 0$. As choices for $\varepsilon$ we suggest using rather large values such as $10^{-2}$. We motivate this by noting that our objective function $\lVert \L_2 \vectorize(\alpha) \rVert_2^2$ is only an approximation of the nonlinear LS error introduced in \eqref{eq:nonlinearLSerror}. 
	Hence, performing many ALS iterations may not always be worth the effort as the effect on the nonlinear error may be unclear. We demonstrate in Section~\ref{sec:trigexample3} that our proposed choice of $\varepsilon = 10^{-2}$ indeed strikes a good balance between computational efficiency and approximation quality. As usual, however, the ideal choice of $\varepsilon$ is problem dependent and we merely share a suggestion which is based on our moderately sized set of numerical experiments.
	
	\subsection{Efficient Construction of Contracted Loewner Matrices}
	\label{sec:contractedLconstruction}
	We note that in the ALS procedure outlined in Algorithm~\ref{alg:als} the contracted Loewner matrices $\L_d^{(j)}$ for $j=1,\ldots,d$ need to be constructed in each iteration. Therefore, one important prerequisite for ALS to be effective is that the contracted matrices can be assembled efficiently. In particular, we would like to construct these matrices without ever fully forming the Loewner matrix $\L_d$. In order to achieve this goal let us inspect the expression
	\begin{align*}
		\L_2^{(1)} = \L_2 \bJ^{(1)} &= \left( \diag\left(\vectorize(\D)\right) \left[\cC^{(1)} \otimes \cC^{(2)}\right]^\top - \left[\cC^{(1)} \otimes \cC^{(2)}\right]^\top \diag(\vectorize(\H)) \right) \bJ^{(1)} \\
		&= \underbrace{\diag\left(\vectorize(\D)\right) \left[\cC^{(1)} \otimes \cC^{(2)}\right]^\top \bJ^{(1)}}_{=\bL^{(1)}} - \underbrace{\left[\cC^{(1)} \otimes \cC^{(2)}\right]^\top \diag(\vectorize(\H)) \bJ^{(1)}}_{=\bR^{(1)}}
	\end{align*}
    for the case $d=2$. In order to construct $\bL^{(1)}$ we follow the steps below:
	\begin{enumerate}
		\item Form the matrix $\bJ^{(1)}$ and reshape it into a tensor $\mathcal{J}^{(1)} \in \C^{n_1 \times n_2 \times n_1 r}$.
		\item Compute $\mathcal{L}^{(1)} = \mathcal{J}^{(1)} \times_1 \cC^{(1)} \times_2 \cC^{(2)} \in \C^{N_1 \times N_2 \times n_1 r}$, where $\times_1$ and $\times_2$ are the mode-$1$ and mode-$2$ tensor-matrix products, respectively \cite{kolda_tensor_2009}.
		\item Reshape $\mathcal{L}^{(1)}$ into a matrix $\hat{\bL}^{(1)} \in \C^{N_1 N_2 \times n_1 r}$ and efficiently compute $\bL^{(1)} = \diag(\vectorize(\D)) \hat{\bL}^{(1)}$ by scaling the $N_1 N_2$ rows of $\hat{\bL}^{(1)}$ with the entries of $\vectorize(\D) \in \C^{N_1 N_2}$.
	\end{enumerate}
	Similarly, we can compute $\bR^{(1)}$ by switching some of the steps outlined above. We then obtain $\L_2^{(1)}$ by subtracting $\bL^{(1)}$ and $\bR^{(2)}$. The procedure for computing $\L_2^{(2)}$ is almost identical and only requires replacing $\bJ^{(1)}$ with $\bJ^{(2)}$. The main idea here is to utilize mode-$k$ tensor-matrix products in order to avoid forming the Kronecker products of Cauchy matrices which are in $\C^{N_1 N_2 \times n_1 n_2}$ explicitly. This keeps the computational cost and memory requirements for forming $\L_2^{(1)}$ and $\L_2^{(2)}$ low. This approach directly extends to the $d$ variable case by considering the matrices in \eqref{eq:Kmanyvariables} and the additional Kronecker products present in $\L_d$.
	
	\section{Numerical Examples}
    \label{sec:numericalexamples}
	In this section we explore the performance of low-rank p-AAA as depicted in Algorithm~\ref{alg:lrpaaa} on various benchmark problems. We make comparisons to the standard p-AAA algorithm and discuss how the choices of input parameters to the low-rank p-AAA algorithm affect the approximation quality. All computations are performed on a compute cluster using Ubuntu 22.04.5 LTS (x86\_64), two Intel(R) Xeon(R) Gold 6246R CPUs, 384 GB RAM and MATLAB R2023b version 23.2.0.2365128. The code which was used for all examples is available in \cite{balicki2025code}. \\
	
	\subsection{Synthetic Parametric Model}
    \label{sec:synthetic}    
	In our first example we consider the rational transfer function $\bH(s,p)$, originating from a synthetic parametric model \cite{the_morwiki_community_synthetic_2005},
    % \begin{equation*}
        % \bH(s,p) = \bc^\top (s\bI - (p \bA_1 + \bA_0))^{-1} \bb,
    % \end{equation*}
    with order $100$ in both $s$ and $p$. Here we have $p \in \R$, hence we are considering the two-variable case where $\bz = (s,p) \in \C^2$ and $\bff(\bz) = \bH(s,p)$. We use this simple example to showcase the effect of various choices for $r$ and to make comparisons to the original p-AAA algorithm. For our numerical experiments, we consider $500$ logarithmically distributed sampling points in $[1,10^4] \dot{\imath}$ for the $s$-variable and $50$ logarithmically spaced points in $[10^{-1.5},1]$ in the parameter $p$. Here, we use the ALS initialization strategy discussed in Section~\ref{sec:alsinit} and use a stopping criterion based on a relative change of $10^{-2}$ for ALS as discussed in Section~\ref{sec:alsstopping}. First, we compare how different choices for the size $r$ of the CP decomposition of $\alpha$ effect the convergence behavior of low-rank p-AAA. For this, we perform $70$ iterations of Algorithm~\ref{alg:lrpaaa} with $r\in\{4,7,10\}$ and compare the resulting relative nonlinear LS error
    \begin{equation}
        \label{eq:nonlinrellserror}
        \sum_{i_1=1}^{N_1} \sum_{i_2=1}^{N_2} \left( \bff\left(Z^{(1)}_{i_1},Z^{(2)}_{i_2}\right) - \br\left(Z^{(1)}_{i_1},Z^{(2)}_{i_2}\right) \right)^2 \Bigg/ \sum_{i_1=1}^{N_1} \sum_{i_2=1}^{N_2} \bff\left(Z^{(1)}_{i_1},Z^{(2)}_{i_2}\right)^2
    \end{equation}
    and relative maximum error
    \begin{equation*}
        \max_{i_1,i_2} \left\lvert \bff\left(Z^{(1)}_{i_1},Z^{(2)}_{i_2}\right) - \br\left(Z^{(1)}_{i_1},Z^{(2)}_{i_2}\right) \right\rvert  \bigg/ \max_{i_1,i_2} \left\lvert \bff\left(Z^{(1)}_{i_1},Z^{(2)}_{i_2}\right) \right\rvert
    \end{equation*}    
    with the errors obtained from $70$ iterations of the original p-AAA algorithm. {These errors are depicted in Figure~\ref{fig:syntheticsampleerror}. Aside from considering the error over the sample data, we also monitor the error over a validation set throughout each iteration. The validation set covers the same intervals as the sample data but is based on $1000$ sampling points in the $s$ and $100$ sampling points in the $p$ variable. The corresponding validation errors for this example are illustrated in Figure~\ref{fig:syntheticvalidationerror}. As one would expect, there is a clear trend: Higher values for $r$ yield better approximations and closer results to the original p-AAA. Additionally, we can see in Figure~\ref{fig:syntheticsampleerror} that the errors on the sample data set are only noticeably lower for the standard p-AAA approximant compared to the low-rank approximants in later iterations. However, as illustrated in Figure~\ref{fig:syntheticvalidationerror}, we observe far smaller gaps between the errors over the validation set. In particular, the validation errors for the low-rank approximants are only marginally larger then for the p-AAA approximant in this experiment. This illustrates the potential for saving memory and computational resources when using low-rank approximations for the barycentric coefficient tensor.} Another important observation is that the difference between error magnitudes on the sample data and the validation data are small for low-rank p-AAA and large for the original p-AAA algorithm. This shows that imposing a low-rank structure on the barycentric coefficients has a regularizing effect on the approximation procedure. Regularization techniques in data-driven modeling are often practically useful, highlighting an additional interesting use case of our proposed algorithm that goes beyond the possible computational savings. We also note that in the two variable case p-AAA and low-rank p-AAA are equivalent if $r \leq \min(n_1,n_2)$. In particular, ALS is guaranteed to converge to the true LS minimum in two steps. Therefore, the errors for low-rank p-AAA deviate from the p-AAA error only in later iterations. Additionally, we note that the final orders of the rational approximants are $(66,15)$ for the standard p-AAA algorithm and $(68,24)$, $(67,24)$ and $(66,24)$ for low-rank p-AAA with $r \in \{4,7,10\}$ (orders are sorted from lowest to highest value for $r$). The orders differ because in some iterations only one of the interpolation sets $\blambda^{(1)}$ or $\blambda^{(2)}$ increased instead of both and greedy selections for the interpolation points varied across the different setups.
    \begin{figure}[t]
    \centering
    \begin{subfigure}[b]{.5\textwidth}
    \begin{tikzpicture}
      \begin{axis}[%
        width  = 0.75\linewidth,
        height = .15\textheight,
        scale only axis,
        xmin = 0,
        xmax = 71,
        ymin = 1e-18,
        ymax = 1e3,
        xlabel = {Iteration Number},
        ylabel = {Relative LS Error},
        ymode = log,
        cycle list name = synthetic,
        legend pos = south west,
        legend cell align={left},
      ]%
      \pgfplotstableread{figure_data/synthetic_parametric_ls_training_errors.dat}\tableINPUT
      \addplot table[x index = 0, y index = 1] from \tableINPUT;
      \addplot table[x index = 0, y index = 2] from \tableINPUT;
      \addplot table[x index = 0, y index = 3] from \tableINPUT;
      \addplot table[x index = 0, y index = 4] from \tableINPUT;
      \legend{p-AAA, $r=4$, $r=7$, $r=10$}
      \end{axis}
   \end{tikzpicture}
   \end{subfigure}%
   \begin{subfigure}[b]{.5\textwidth}
    \begin{tikzpicture}
      \begin{axis}[%
        width  = 0.75\linewidth,
        height = .15\textheight,
        scale only axis,
        xmin = 0,
        xmax = 71,
        ymin = 1e-9,
        ymax = 1e2,
        xlabel = {Iteration Number},
        ylabel = {Relative Max Error},
        ymode = log,
        cycle list name = synthetic,
        legend pos = south west,
        legend cell align={left},
      ]%
      \pgfplotstableread{figure_data/synthetic_parametric_max_training_errors.dat}\tableINPUT
      \addplot table[x index = 0, y index = 1] from \tableINPUT;
      \addplot table[x index = 0, y index = 2] from \tableINPUT;
      \addplot table[x index = 0, y index = 3] from \tableINPUT;
      \addplot table[x index = 0, y index = 4] from \tableINPUT;
      \legend{p-AAA, $r=4$, $r=7$, $r=10$}
      \end{axis}
   \end{tikzpicture}
   \end{subfigure}%
    \caption{Sample data errors for example~\ref{sec:synthetic}. The relative (nonlinear) LS and relative maximum errors are depicted for the standard p-AAA algorithm and low-rank p-AAA with various choices for $r$.}
    \label{fig:syntheticsampleerror}
    \end{figure}
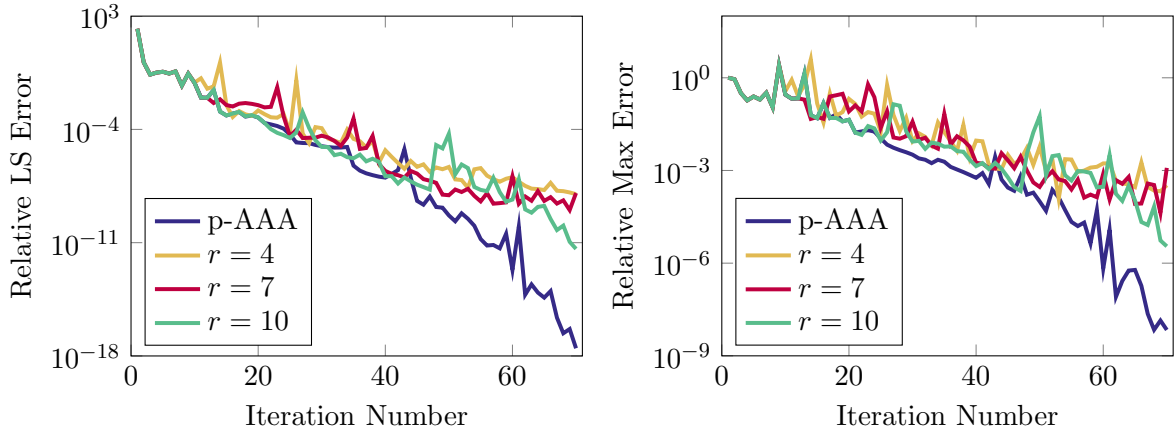

    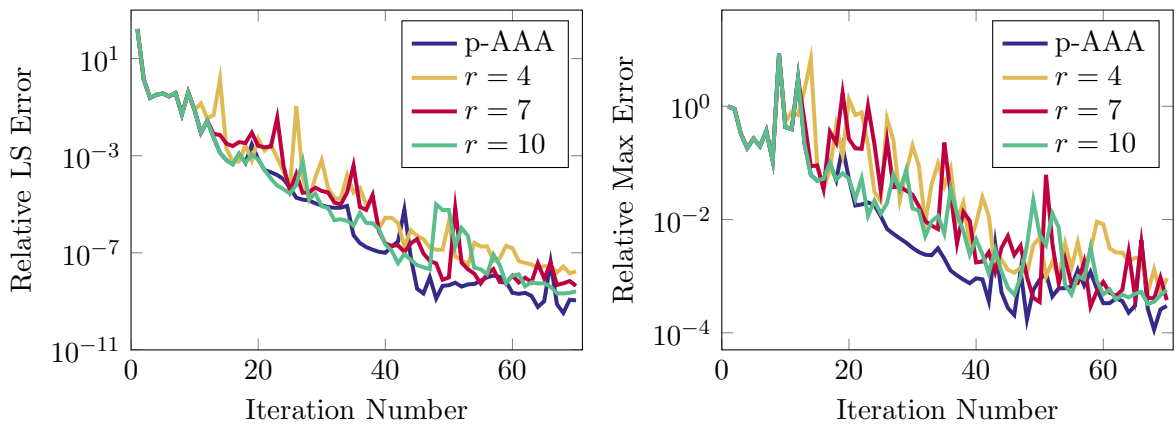
\begin{figure}[t]
    \centering
    \begin{subfigure}[b]{.5\textwidth}
    \begin{tikzpicture}
      \begin{axis}[%
        width  = 0.75\linewidth,
        height = .15\textheight,
        scale only axis,
        xmin = 0,
        xmax = 71,
        ymin = 1e-11,
        ymax = 1e3,
        xlabel = {Iteration Number},
        ylabel = {Relative LS Error},
        ymode = log,
        cycle list name = synthetic,
        legend pos = north east,
        legend cell align={left},
      ]%
      \pgfplotstableread{figure_data/synthetic_parametric_ls_validation_errors.dat}\tableINPUT
      \addplot table[x index = 0, y index = 1] from \tableINPUT;
      \addplot table[x index = 0, y index = 2] from \tableINPUT;
      \addplot table[x index = 0, y index = 3] from \tableINPUT;
      \addplot table[x index = 0, y index = 4] from \tableINPUT;
      \legend{p-AAA, $r=4$, $r=7$, $r=10$}
      \end{axis}
   \end{tikzpicture}
   \end{subfigure}%
   \begin{subfigure}[b]{.5\textwidth}
    \begin{tikzpicture}
      \begin{axis}[%
        width  = 0.75\linewidth,
        height = .15\textheight,
        scale only axis,
        xmin = 0,
        xmax = 71,
        ymin = 5e-5,
        ymax = 5e1,
        xlabel = {Iteration Number},
        ylabel = {Relative Max Error},
        ymode = log,
        cycle list name = synthetic,
        legend pos = north east,
        legend cell align={left},
      ]%
      \pgfplotstableread{figure_data/synthetic_parametric_max_validation_errors.dat}\tableINPUT
      \addplot table[x index = 0, y index = 1] from \tableINPUT;
      \addplot table[x index = 0, y index = 2] from \tableINPUT;
      \addplot table[x index = 0, y index = 3] from \tableINPUT;
      \addplot table[x index = 0, y index = 4] from \tableINPUT;
      \legend{p-AAA, $r=4$, $r=7$, $r=10$}
      \end{axis}
   \end{tikzpicture}
   \end{subfigure}%
    \caption{Validation data errors for example~\ref{sec:synthetic}. The relative (nonlinear) LS and relative maximum errors are depicted for the standard p-AAA algorithm and low-rank p-AAA with various choices for $r$.}
    \label{fig:syntheticvalidationerror}
    \end{figure}

	\subsection{Approximation of a Stationary Thermal Model}
	\label{sec:thermalblock}	
	Next, we consider approximation of a parametrized stationary thermal model. The model has four parameters $\bz = \bp=\big(p^{(1)},p^{(2)},p^{(3)},p^{(4)}\big) \in \R^4$ that define the heat conductivity in distinct subdomains. While various forms of this so-called ``cookie baking" benchmark problem exist, we focus here on the finite element discretization introduced in \cite{rave_non-stationary_2021}. In particular, we consider a stationary version of the problem given by
	\begin{equation}
		\begin{aligned}
			\nabla \cdot (\sigma(x,\bp) \nabla u(x, \bp) ) = 0 \quad & \text{for} \quad  x \in \Omega, \\
			\sigma(x,\bp) \nabla u(x, \bp) \cdot n(x) = 1 \quad & \text{for} \quad x \in \Gamma_L, \\
			\sigma(x,\bp) \nabla u(x, \bp) \cdot n(x) = 0 \quad & \text{for} \quad x \in \Gamma_T \cup \Gamma_B, \\
			u(x, \bp) = 0 \quad & \text{for} \quad x \in \Gamma_R, \\
		\end{aligned}
	\end{equation}
	where $\Gamma_L,\Gamma_T,\Gamma_R,\Gamma_B$ are the left, top, right and bottom boundaries of $\Omega = (0,1)^2$, respectively, and $n(x)$ is the normal vector on the corresponding boundary. In this model the heat conductivity $\sigma(x,\bp)$ is given by
	\begin{equation*}
		\sigma(x,\bp) = \begin{cases}
			p^{(j)} \quad &\text{if} \quad x \in \Omega_j, \\
			1 \quad &\text{otherwise},
		\end{cases}
	\end{equation*}
	where $\Omega_1,\ldots,\Omega_4$ are disjoint circular subdomains of $\Omega$. Additionally, we introduce the model output
	\begin{equation*}
		\by(\bp) = \int_{\Omega} u(x, \bp) \, dx
	\end{equation*}
	as the average temperature over the entire domain. Discretizing this model via the finite element method yields the representation
	\begin{align*}
		\bA(\bp) \bx(\bp) &= \bb \\
		\bff(\bp) &= \bc^\top \bx(\bp)
	\end{align*}
	introduced in \eqref{eq:stationaryparametric}. In the following we consider the same sets of sampling points $\bZ^{(1)}=\cdots=\bZ^{(4)}$ for the four parameters, where each $\bZ^{(j)}$ contains $40$ logarithmically spaced values between $10^{-6}$ and $10^2$. We execute Algorithm~\ref{alg:lrpaaa} with $r=5$ and use a relative change of the ALS objective function of less than $10^{-2}$ as a stopping criterion for ALS. After $8$ iterations we obtain a rational approximant of order $(7,7,6,6)$ with a relative maximum error of approximately $7.32 \times 10^{-4}$ on the set of samples. The approximation error for a set of validation points outside of the sampled data set is illustrated in Figure~\ref{fig:thermalblock}. Note that for this problem we can write $\bA(\bp)$ in a block tridiagonal form which closely resembles the structure for $\bK(\bz)$ introduced in \eqref{eq:blockK}. This means that an exact representation of $\bff$ in terms of a barycentric form (computed, e.g., via the parametric Loewner framework) would have barycentric coefficients that could be represented exactly by low-rank tensors. Here we did not exactly recover $\bff$ but point out that this connection may have contributed to the high accuracy of the approximant computed by low-rank p-AAA. Further, we note that the standard p-AAA algorithm was not able to produce a better approximant before running out of memory due to the size of the underlying Loewner matrix. Specifically, the LS errors were several orders of magnitudes larger for the original p-AAA algorithm throughout the first few iterations. This behavior could be due to numerical inaccuracies which are caused by severely ill-conditioned Loewner matrices, which exceed a condition number of $10^{20}$ in multiple steps of the algorithm for data originating from this model. In this example the contracted Loewner matrices that are formed in the ALS procedure are much better behaved with condition numbers closer to $10^{10}$ for most iterations and therefore do not cause noticeable numerical errors. We note that we did not observe such a large discrepancy in condition numbers between the low-rank and standard versions of p-AAA in example~\ref{sec:synthetic}. It therefore remains an open question under which conditions we can expect the low-rank p-AAA LS matrices to be better conditioned than the Loewner matrices occurring in the original algorithm.
    \begin{figure}
    \centering
    \begin{subfigure}[b]{.5\textwidth}
    \begin{tikzpicture}
    \begin{axis}[
        width  = .95\linewidth,
        height = .225\textheight,
        view={135}{25},
        colormap name=paaacmap,
        grid=major,
        grid style={gray!25},
        tick style={draw=none},
        xmin=1e-6,
        xmax=1e2,
        ymin=1e-6,
        ymax=1e2,
        xtick={1e-4, 1},
        ytick={1e-4, 1},
        xmode=log,
        ymode=log,
        xlabel={$p^{(1)}$},
        ylabel={$p^{(2)}$},
        zlabel={$\bff(\bp)$},
        mesh/ordering=y varies,
        mesh/rows=34,
    ]
    
    \addplot3[surf] file {figure_data/thermal_block_validation_data.dat};%
    
    \end{axis}
    \end{tikzpicture}
    \end{subfigure}%
    \begin{subfigure}[b]{.5\textwidth}
    \begin{tikzpicture}
    \begin{axis}[
        width  = .95\linewidth,
        height = .225\textheight,
        view={135}{25},
        colormap name=paaacmap,
        grid=major,
        grid style={gray!25},
        tick style={draw=none},
        xmin=1e-6,
        xmax=1e2,
        ymin=1e-6,
        ymax=1e2,
        zmin=1e-12,
        zmax=1e-5,
        xtick={1e-4, 1},
        ytick={1e-4, 1},
        xmode=log,
        ymode=log,
        zmode=log,
        xlabel={$p^{(1)}$},
        ylabel={$p^{(2)}$},
        zlabel={$\be(\bp)$},
        mesh/ordering=y varies,
        mesh/rows=34,
    ]
    
    \addplot3[surf] file {figure_data/thermal_block_validation_error.dat};%
    
    \end{axis}
    \end{tikzpicture}
    \end{subfigure}
    \caption{Surface plot for $\bff$ and the error $\be := \lvert \bff - \br \rvert$ for example~\ref{sec:thermalblock}. Here, we consider $p^{(3)} = 10^{-7}$ and $p^{(4)} = 10^{3}$ which are an order of magnitude below/above the sampling points used for computing $\br$.}
    \label{fig:thermalblock}
    \end{figure}
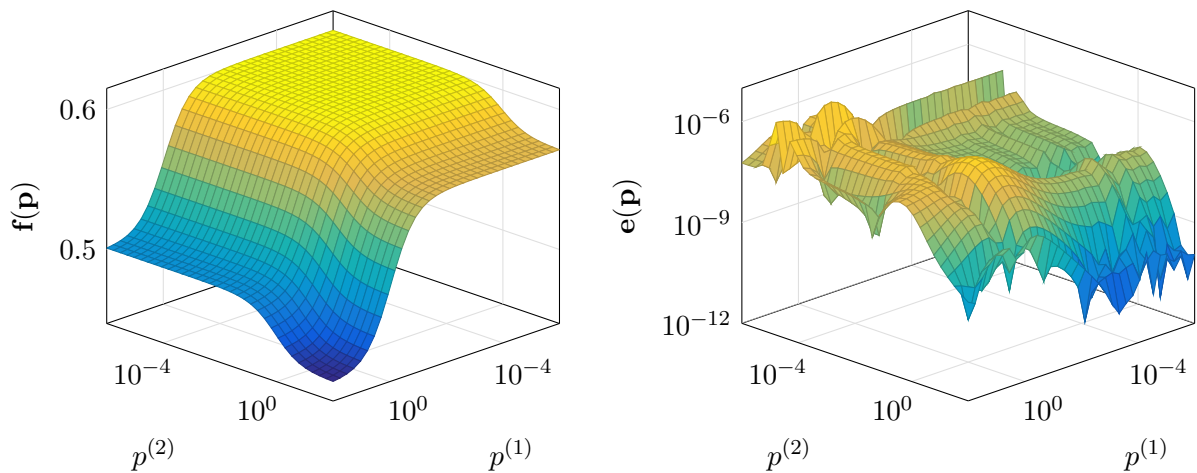
    
	\subsection{Trigonometric Example}
	\label{sec:trigexample}
	Next, we consider approximation of the irrational function
	\begin{equation}
		\label{eq:trigf}
		\bff(\bz) = \frac{z^{(1)} + \cdots + z^{(d)}}{2d + \cos(z^{(1)}) + \cdots + \cos(z^{(d)})}
	\end{equation}
	on $[-a,a]^d$. We consider two separate setups for this problem, one with $d=3$ and $a=10$ and another with $d=5$ and $a=4$.

\subsubsection{Trigonometric Example with Three Variables}
    \label{sec:trigexample3}
\begin{figure}
    \centering
    \begin{subfigure}[b]{.5\textwidth}
    \begin{tikzpicture}
    \begin{axis}[
        width  = .95\linewidth,
        height = .225\textheight,
        view={-45}{25},
        colormap name=paaacmap,
        grid=major,
        grid style={gray!25},
        tick style={draw=none},
        xmin=-10,
        xmax=10,
        ymin=-10,
        ymax=10,
        zmin=-4,
        zmax=4,
        xtick={-5,5},
        ytick={-5,5},
        xlabel={$z^{(1)}$},
        ylabel={$z^{(2)}$},
        zlabel={$\bff(\bz)$},
        mesh/ordering=x varies,
        mesh/rows=40,
    ]
    
    \addplot3[surf] file {figure_data/trig_3_validation_data.dat};%
    
    \end{axis}
    \end{tikzpicture}
    \end{subfigure}%
    \begin{subfigure}[b]{.5\textwidth}
    \begin{tikzpicture}
    \begin{axis}[
        width  = .95\linewidth,
        height = .225\textheight,
        view={-45}{25},
        colormap name=paaacmap,
        grid=major,
        grid style={gray!25},
        tick style={draw=none},
        xmin=-10,
        xmax=10,
        ymin=-10,
        ymax=10,
        zmin=1e-11,
        zmax=1e-6,
        zmode=log,
        xtick={-5,5},
        ytick={-5,5},
        xlabel={$z^{(1)}$},
        ylabel={$z^{(2)}$},
        zlabel={$\be(\bz)$},
        mesh/ordering=x varies,
        mesh/rows=40,
    ]
    
    \addplot3[surf] file {figure_data/trig_3_validation_error.dat};%
    
    \end{axis}
    \end{tikzpicture}
    \end{subfigure}
    \caption{Surface plot for $\bff$ and the error $\be := \lvert \bff - \br \rvert$ for example~\ref{sec:trigexample3}. The figure depicts $\bff$ and $\be$ with the fixed value $z^{(3)} = 0$.}
    \label{fig:trigexample_r3_tol_1e-02}
    \end{figure}
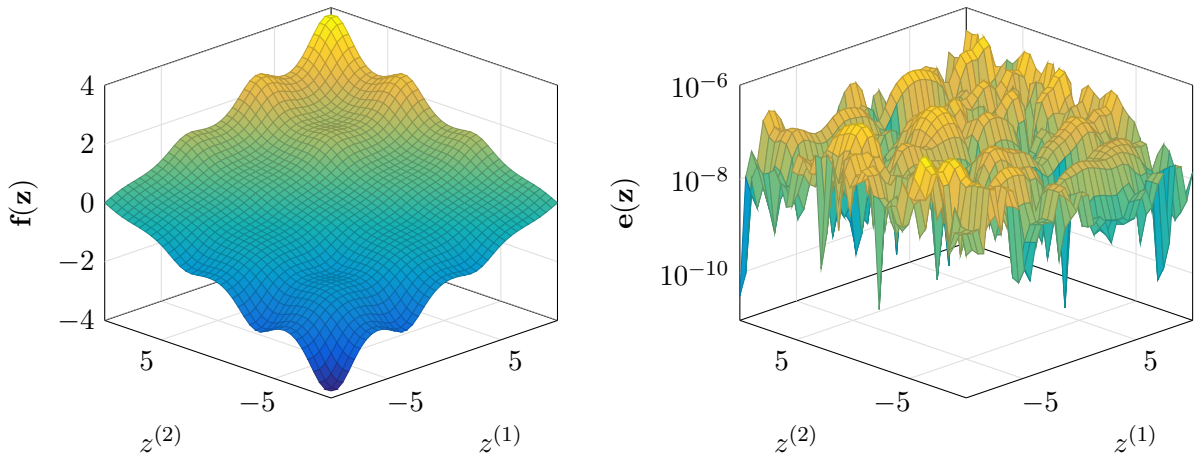
    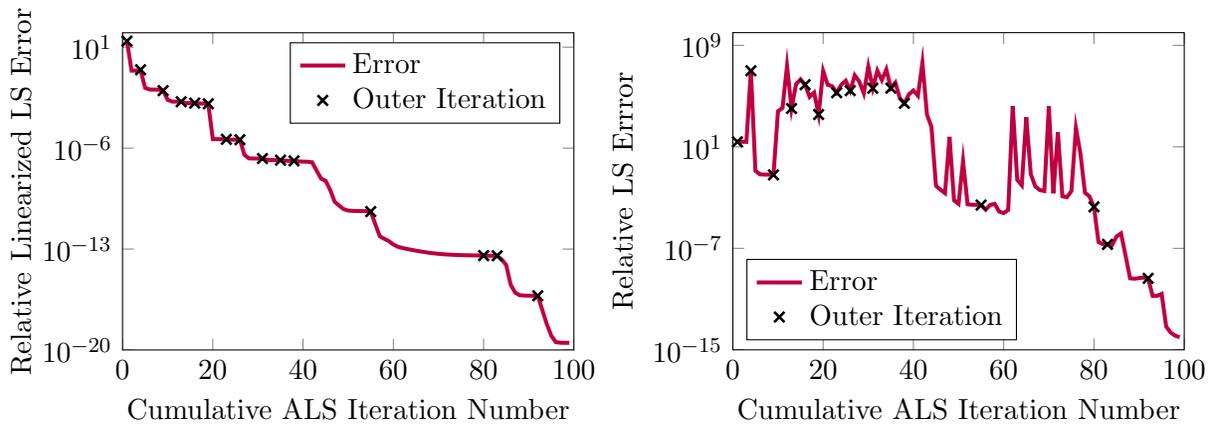
\begin{figure}[t]
    \centering
    \begin{subfigure}[b]{.5\textwidth}
    \begin{tikzpicture}
      \begin{axis}[%
        width  = 0.75\linewidth,
        height = .14\textheight,
        scale only axis,
        xmin = 0,
        xmax = 100,
        ymin = 1e-20,
        ymax = 1e2,
        xminorticks = false,
        yminorticks = false,
        xlabel = {Cumulative ALS Iteration Number},
        ylabel = {Relative Linearized LS Error},
        ylabel style   = {yshift = -.3em},
        ymode = log,
        cycle list name    = trigLSerrs,
        legend pos=north east,
        legend cell align={left},
        unbounded coords=jump,
      ]%
      \pgfplotstableread{figure_data/trig_3_convergence.dat}\tableINPUT
      \addplot table[x index = 0, y index = 1] from \tableINPUT;
      \pgfplotstableread{figure_data/trig_3_convergence_outer.dat}\outertableINPUT
      \addplot[only marks, mark=x, mark size=2.75pt, line width=1pt] table[x index = 0, y index = 1] from \outertableINPUT;
      \legend{Error, Outer Iteration}
      \end{axis}
   \end{tikzpicture}
   \end{subfigure}%
    \begin{subfigure}[b]{.5\textwidth}
    \begin{tikzpicture}
      \begin{axis}[%
        width  = 0.75\linewidth,
        height = .14\textheight,
        scale only axis,
        xmin = 0,
        xmax = 100,
        ymin = 1e-15,
        ymax = 1e10,
        xminorticks = false,
        yminorticks = false,
        xlabel = {Cumulative ALS Iteration Number},
        ylabel = {Relative LS Errror},
        ymode = log,
        cycle list name = trigLSerrs,
        legend pos=south west,
        legend cell align={left},
        unbounded coords=jump,
      ]%

      \pgfplotstableread{figure_data/trig_3_convergence.dat}\tableINPUT
      \addplot table[x index = 0, y index = 2] from \tableINPUT;
      \pgfplotstableread{figure_data/trig_3_convergence_outer.dat}\outertableINPUT
      \addplot[only marks, mark=x, mark size=2.75pt, line width=1pt] table[x index = 0, y index = 2] from \outertableINPUT;
      \legend{Error, Outer Iteration}
      \end{axis}
    \end{tikzpicture}
    \end{subfigure}
    \caption{Linearized and nonlinear LS approximation errors for example~\ref{sec:trigexample3}. The errors are depicted for each ALS iteration step performed throughout the low-rank p-AAA algorithm. Hence, each cumulative ALS iteration depicted on the $x$-axis represents one complete loop of Algorithm~\ref{alg:als}. Additionally, the outer iterations (i.e., the iterations of Algorithm~\ref{alg:lrpaaa}) where new interpolation points were added through a greedy selection are indicated via markers in the plots. The final order of the rational approximant is $(12,12,13)$ for this example. }
    \label{fig:trigexample_r3_tol_1e-02_convergence}
    \end{figure}%

	First, we consider this problem for $a=10$ and $d=3$. We use $100$ linearly spaced points in $[-10,10]$ as the sampling points for each variable. We then execute low-rank p-AAA with $r=3$. For the ALS stopping criterion we use a relative change in the objective function $\lVert \L_d \vectorize(\alpha) \rVert_2^2$ of less than $10^{-2}$. With these settings only $15$ iterations of low-rank p-AAA were necessary to compute a highly accurate rational approximant. The error of the approximant is illustrated Figure~\ref{fig:trigexample_r3_tol_1e-02}. In each iteration ALS required between two and $25$ iterations before the stopping criterion was satisfied. For this setup we monitor the relative linearized LS error
    \begin{align*}
        \sum_{i_1=1}^{N_1} \sum_{i_2=1}^{N_2} \sum_{i_3=1}^{N_3} \bigg( & \bd\left(Z^{(1)}_{i_1},Z^{(2)}_{i_2},Z^{(3)}_{i_3}\right)\bff\left(Z^{(1)}_{i_1},Z^{(2)}_{i_2},Z^{(3)}_{i_3}\right) \\
        &\qquad - \bn\left(Z^{(1)}_{i_1},Z^{(2)}_{i_2},Z^{(3)}_{i_3}\right) \bigg)^2 \Bigg/ \sum_{i_1=1}^{N_1} \sum_{i_2=1}^{N_2} \sum_{i_3=1}^{N_3}\bff\left(Z^{(1)}_{i_1},Z^{(2)}_{i_2},Z^{(3)}_{i_3}\right)^2
    \end{align*}
    as well as the relative nonlinear LS error along the lines of \eqref{eq:nonlinrellserror} for each step of ALS throughout all low-rank p-AAA iterations. We see in Figure~\ref{fig:trigexample_r3_tol_1e-02_convergence} that the last three p-AAA iterations had the most significant impact on lowering the nonlinear LS error. Further, we observe that the true error is very sensitive to changes in the barycentric coefficients and small changes in the linearized error may result in large changes in the non-linear error (see for example the error behavior between iterations $55$ and $80$). Additionally, finding a better solution for the linearized LS problem does not always result in a lower nonlinear error (again, observe iterations $55$ and $80$). These observations motivate using rather large stopping tolerances for ALS such that not too much time is spent on refining an ALS solution which may not significantly improve the nonlinear approximation error. This saves significant computational time in the first few iterations of low-rank p-AAA. Furthermore, we emphasize that the monotonic decrease in the linearized error is entirely due to our ALS initialization strategy introduced in Section~\ref{sec:alsinit}. If we use random initial guesses, the linearized errors will be significantly larger at the start of each ALS iteration, which typically leads to more required generalized SVD computations before the stopping criterion is satisfied.

	\subsubsection{Trigonometric Example with Five Variables}
    \label{sec:trigexample5}
	Next we consider approximation of $\bff$ defined in $\eqref{eq:trigf}$ with $d=5$ on $[-4,4]^5$. We choose $30$ linearly spaced points in $[-4,4]$ as sampling points for all variables. Additionally, we use $r=2$ and a stopping criterion based on a relative objective function change of $10^{-2}$ for ALS. We use a relative maximum error of $10^{-3}$ as a stopping criterion for low-rank p-AAA and obtain a rational approximant of order $(14,13,14,14,14)$ after $24$ iterations. We see in Figure~\ref{fig:trigexample_d5} the high fidelity of the approximation for a specific set of values which were not included in the sampled data set. We note that the standard p-AAA algorithm does not converge on a regular laptop computer for this example because the Loewner matrices require untenable memory after a few steps of the algorithm. For example, storage of the Loewner matrix $\L_5$ in double precision needed for computing the barycentric coefficients of an order $(5,5,5,5,5)$ approximant would require roughly $495$GB of memory  for this example.
        
    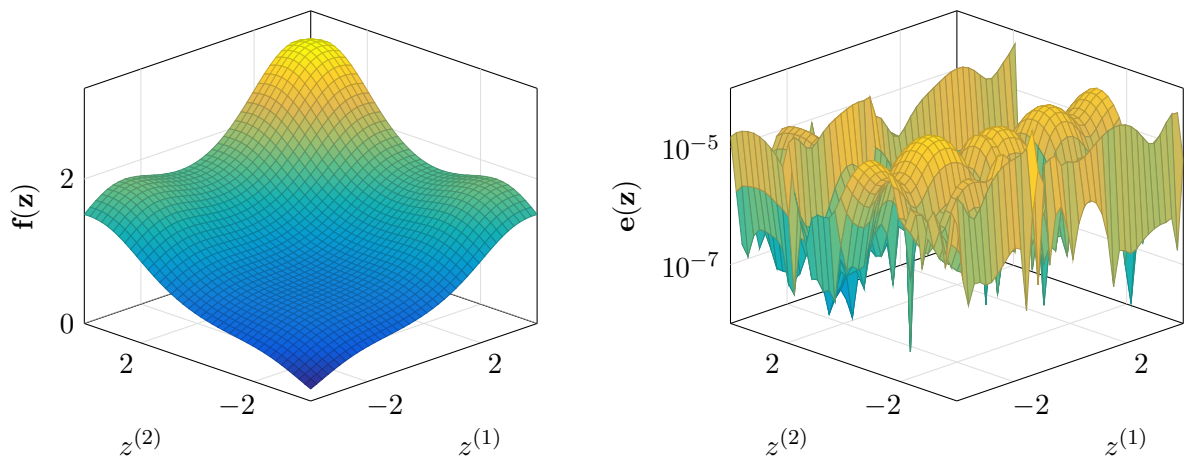
\begin{figure}
    \centering
    \begin{subfigure}[b]{.5\textwidth}
    \begin{tikzpicture}
    \begin{axis}[
        width  = .95\linewidth,
        height = .225\textheight,
        view={-45}{25},
        colormap name=paaacmap,
        grid=major,
        grid style={gray!25},
        tick style={draw=none},
        xmin=-4,
        xmax=4,
        ymin=-4,
        ymax=4,
        xtick={-2,2},
        ytick={-2,2},
        zmin=0,
        xlabel={$z^{(1)}$},
        ylabel={$z^{(2)}$},
        zlabel={$\bff(\bz)$},
        mesh/ordering=x varies,
        mesh/rows=40,
    ]
    
    \addplot3[surf] file {figure_data/trig_5_validation_data.dat};%
    
    \end{axis}
    \end{tikzpicture}
    \end{subfigure}%
    \begin{subfigure}[b]{.5\textwidth}
    \begin{tikzpicture}
    \begin{axis}[
        width  = .95\linewidth,
        height = .225\textheight,
        view={-45}{25},
        colormap name=paaacmap,
        grid=major,
        grid style={gray!25},
        tick style={draw=none},
        xmin=-4,
        xmax=4,
        ymin=-4,
        ymax=4,
        zmin=1e-8,
        zmax=1e-4,
        zmode=log,
        xtick={-2,2},
        ytick={-2,2},
        xlabel={$z^{(1)}$},
        ylabel={$z^{(2)}$},
        zlabel={$\be(\bz)$},
        mesh/ordering=x varies,
        mesh/rows=40,
    ]
    
    \addplot3[surf] file {figure_data/trig_5_validation_error.dat};%
    
    \end{axis}
    \end{tikzpicture}
    \end{subfigure}
    \caption{Surface plot for $\bff$ and the error $\be := \lvert \bff - \br \rvert$ for example~\ref{sec:trigexample5}. The figure depicts $\bff$ and $\be$ with the fixed values $z^{(3)} = 2.5$, $z^{(4)} =3$ and $z^{(5)} =3.5$.}
    \label{fig:trigexample_d5}
    \end{figure}
	
	\subsection{Approximation of a Mass Spring Damper Model}
    \label{sec:massspringdamper}
	As a final example we consider a modified version of the mass-spring-damper model from \cite{gugercin_structure-preserving_2012}. In this model $n$ masses are connected in a chain with springs serving as a connectors between two adjacent masses. The mass on the right-hand side of the chain is connected to a fixed wall via a spring and an external force is acting as a model input on the free-floating mass on the left-hand side of the chain. Additionally, each mass is connected to a damper. We consider the velocity of the first mass as a model output. Here, we assume that the masses and damping coefficients are all equal to a fixed value but that the stiffness of the springs is parametrized. In particular we assign the stiffness $k^{(1)}$ to the first $n/4$ springs, $k^{(2)}$ to the next $n/4$ springs and so on. Ultimately we end up with a mass spring damper model that depends on four parameters $\bk = \big(k^{(1)},k^{(2)},k^{(3)},k^{(4)}\big) \in \R^4$. As motivated in Section~\ref{sec:motivatingproblems}, we aim to approximate the transfer function of the underlying model given by
	\begin{equation*}
		\bH(s,\bk) = \bc^\top (s\bI - \bA(\bk))^{-1} \bb,
	\end{equation*}
	using the low-rank p-AAA algorithm. Here we consider $\bz = (s,\bk) \in \C^5$ and therefore an approximation problem with $d=5$ variables. Additionally, we consider a model with $n=40$ masses in this example. For our sample data we use $50$ equidistantly spaced points in $[0.1,2]\dot{\imath}$ in the $s$-variable and $25$ equidistant points in $[0.5,1]$ for the four spring stiffness values. We use $r=3$ and obtain a rational approximant with the order $(25,11,11,11,11)$ after $29$ iterations of the algorithms. Note that for this problem we set the maximum number of interpolation points for each parameter to $12$, which corresponds to an order-$11$ rational function in each of the parameters $k^{(1)},k^{(2)},k^{(3)}$ and $k^{(4)}$. The relative LS error at the end of low-rank p-AAA was approximately $2.18 \times 10^{-7}$. As depicted in Figure~\ref{fig:msd} we obtain high-fidelity approximations for parameter values which are outside of the sampled intervals.
    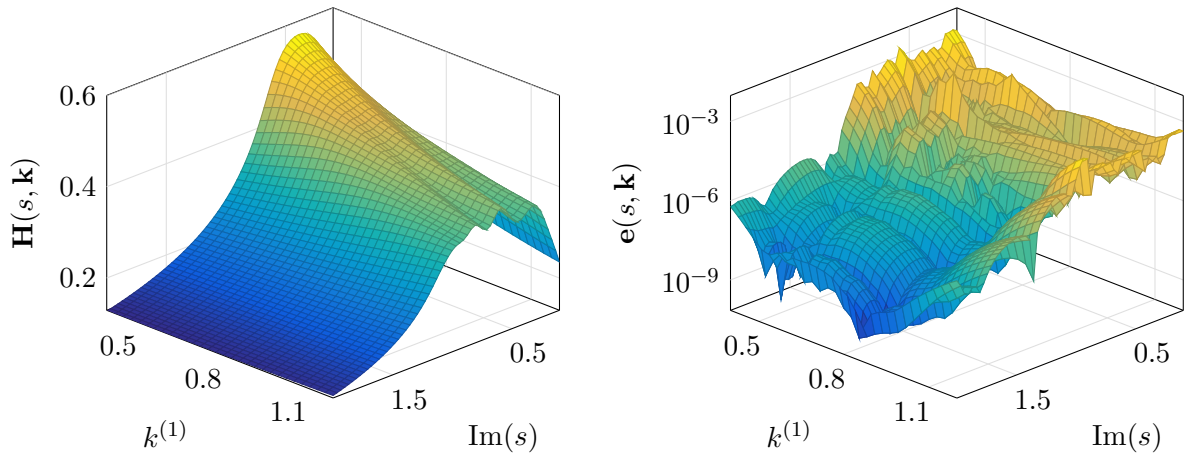
\begin{figure}
    \centering
    \begin{subfigure}[b]{.5\textwidth}
    \begin{tikzpicture}
    \begin{axis}[
        width  = .95\linewidth,
        height = .225\textheight,
        view={135}{30},
        colormap name=paaacmap,
        grid=major,
        grid style={gray!25},
        tick style={draw=none},
        xmin=0.1,
        xmax=2,
        xtick={0.5,1.5},
        ymin=0.4,
        ymax=1.2,
        ytick={0.5,0.8,1.1},
        zmin=0.1294,
        zmax=0.6,
        xlabel={$\Im(s)$},
        ylabel={$k^{(1)}$},
        zlabel={$\bH(s,\bk)$},
        mesh/ordering=x varies,
        mesh/rows=30,
        mesh/cols=60,
    ]
    
    \addplot3[surf] file {figure_data/msd_validation_data.dat};%
    
    \end{axis}
    \end{tikzpicture}
    \end{subfigure}%
    \begin{subfigure}[b]{.5\textwidth}
    \begin{tikzpicture}
    \begin{axis}[
        width  = .95\linewidth,
        height = .225\textheight,
        view={135}{30},
        colormap name=paaacmap,
        grid=major,
        grid style={gray!25},
        tick style={draw=none},
        xmin=0.1,
        xmax=2,
        xtick={0.5,1.5},
        ymin=0.4,
        ymax=1.2,
        ytick={0.5,0.8,1.1},
        ztick={1e-3,1e-6,1e-9},
        zmode=log,
        xlabel={$\Im(s)$},
        ylabel={$k^{(1)}$},
        zlabel={$\be(s,\bk)$},
        mesh/ordering=x varies,
        mesh/rows=30,
        mesh/cols=60,
    ]
    
    \addplot3[surf] file {figure_data/msd_validation_error.dat};%
    
    \end{axis}
    \end{tikzpicture}
    \end{subfigure}
    \caption{Surface plot for $\bH$ and the error $\be := \lvert \bH - \br \rvert$ for example~\ref{sec:massspringdamper}. The figure depicts $\bH$ and $\be$ with the fixed values $k^{(2)}=0.4$, $k^{(3)} = 1.1$ and $k^{(4)} = 0.75$. Note that the sample data for the stiffness values was in the interval $[0.5,1]$. Hence, errors are depicted for several values outside of the sampled intervals.}
    \label{fig:msd}
    \end{figure}

	\section{Conclusion and Future Directions}
    \label{sec:conclusion}

    We introduced barycentric forms for rational functions based on low-rank tensor decompositions of barycentric coefficients and showed how to incorporate them into the p-AAA framework, leading to the low-rank p-AAA algorithm. The proposed algorithm has a significantly improved computational complexity compared to standard p-AAA, making it possible to compute accurate rational approximations for problems that involve a higher number of variables. In our numerical experiments we demonstrate that our approach is effective by tackling several problems which are not computationally feasible in the standard p-AAA framework. This highlights the effectiveness of using function representations based on separable forms in rational approximation frameworks. We note that similar to  original p-AAA, our approach still requires access to the full data tensor $\D$ with $N_1 \cdots N_d$ entries. Removing this factor from the algorithm's computational complexity and enhancing its scalability even more will require revising the barycentric formulation such that it is applicable to non-grid data samples. 
    Such considerations are left to be investigated in future work.

    \section*{Acknowledgments}
    This work was supported in part by the National Science Foundation (NSF), United States under Grant No.
    DMS-2411141.
		
    \bibliographystyle{plainurl}
    \bibliography{references}

    \appendix

    \section{Proof of Lemma~\ref{lemma:rank1data}}
        \label{appendix:proofrank1lemma}
        Let the sampling points $\bZ^{(j)}$ and interpolation nodes $\blambda^{(j)}$ be given such that $\rk(\D) = 1$ and $\L_d \vectorize(\alpha) = 0$ for some $\alpha \neq 0$. The assumption $\L_d \vectorize(\alpha) = 0$ is satisfied if and only 
        \begin{equation*}
            \bd\left(Z^{(1)}_{i_1},\ldots,Z^{(d)}_{i_d}\right) \bff\left(Z^{(1)}_{i_1},\ldots,Z^{(d)}_{i_d}\right)  - \bn\left(Z^{(1)}_{i_1},\ldots,Z^{(d)}_{i_d}\right) = 0
        \end{equation*}
        and therefore
        \begin{equation*}
             \bff\left(Z^{(1)}_{i_1},\ldots,Z^{(d)}_{i_d}\right) = \bn\left(Z^{(1)}_{i_1},\ldots,Z^{(d)}_{i_d}\right) \bigg/ \bd\left(Z^{(1)}_{i_1},\ldots,Z^{(d)}_{i_d}\right),
        \end{equation*}        
        for $i_j=1,\ldots,N_j$, $j=1,\ldots,d$ and $\bd,\bn$ as in \eqref{eq:nd} and \eqref{eq:dd}, respectively. In particular, this means that $\br = \bn / \bd$ recovers the data in $\D$ exactly:
        \begin{equation*}
            \D_{i_1\ldots i_d} = \br\left(Z^{(1)}_{i_1},\ldots,Z^{(d)}_{i_d}\right) \quad \text{for} \quad i_j = 1,\ldots,N_j \quad \text{and} \quad j=1,\ldots,d.
        \end{equation*}
        Further, the assumption $\rk(\D) = 1$ implies that there exist CP factors $\delta^{(j)} \in \C^{N_j}$ for $j=1,\ldots,d$ such that
        \begin{equation*}
            \D_{i_1\ldots i_d} = \delta^{(1)}_{i_1} \cdots \delta^{(d)}_{i_d}  \quad \text{for} \quad i_j = 1,\ldots,N_j \quad \text{and} \quad j=1,\ldots,d.
        \end{equation*}
        Next, we define the functions
        \begin{equation*}
            \br^{(j)}\left(z^{(j)}\right) = \br\left(Z^{(1)}_{1},\ldots,Z^{(j-1)}_1,z^{(j)},Z^{(j+1)}_1,\ldots,Z^{(d)}_{1}\right) \\
        \end{equation*}
        for $j = 1,\ldots,d$. We note that $\br^{(j)}$ satisfies the interpolation conditions
        \begin{equation*}
            \br^{(j)}\left(\lambda^{(j)}_{i_j}\right) = \bff\left(Z^{(1)}_{1},\ldots,Z^{(j-1)}_1,\lambda^{(j)}_{i_j},Z^{(j+1)}_1,\ldots,Z^{(d)}_{1}\right) \quad \text{for} \quad i_j = 1,\ldots,n_j
        \end{equation*}
        and can therefore be written in the barycentric form
        \begin{equation}
            \label{eq:rjbarycentric}
            \br^{(j)}\left(z^{(j)}\right) = \cC^{(j)}\left(z^{(j)}\right)^\top \left( \alpha^{(j)} \circ \H^{(j)} \right) \bigg/ \cC^{(j)}\left(z^{(j)}\right)^\top \alpha^{(j)},
        \end{equation}
        where $\H^{(j)} = \left[ \br^{(j)}\left(\lambda^{(j)}_1\right) \; \cdots \; \br^{(j)}\left(\lambda^{(j)}_{n_j}\right) \right]^\top \in \C^{n_j}$ and $\alpha^{(j)} \in \C^{n_j}$. Next we assume without loss of generality that $\D_{1 \ldots 1} \neq 0$ (note that we can always rearrange the sampling points such that this assumption holds). This allows for writing
        \begin{align*}
        \D_{i_1\ldots i_d} = \delta^{(1)}_{i_1}\delta^{(2)}_{i_2}\cdots\delta^{(d)}_{i_d} 
        &= \frac{\left(\delta^{(1)}_{i_1}\delta^{(2)}_{1}\cdots\delta^{(d)}_{1}\right)\cdots \left(\delta^{(1)}_{1}\delta^{(2)}_{1}\cdots\delta^{(d)}_{i_d}\right)}{\left(\delta^{(1)}_{1}\delta^{(2)}_{1}\cdots \delta^{(d)}_{1}\right)^{d-1}} \\
        &=
         \frac{\br\left(Z^{(1)}_{i_1},Z^{(2)}_{1},\ldots,Z^{(d)}_{1}\right)\cdots \br\left(Z^{(1)}_{1},Z^{(2)}_{1},\ldots,Z^{(d)}_{i_d}\right)}{\br\left(Z^{(1)}_{1},Z^{(2)}_{1},\ldots,Z^{(d)}_{1}\right)^{d-1}} \\
        &= 
         \frac{\br^{(1)}\left(Z^{(1)}_{i_1}\right) \cdots \br^{(d)}\left(Z^{(d)}_{i_d}\right) }{\br\left(Z^{(1)}_{1},Z^{(2)}_{1},\ldots,Z^{(d)}_{1}\right)^{d-1}} = \frac{\br^{(1)}\left(Z^{(1)}_{i_1}\right) \cdots \br^{(d)}\left(Z^{(d)}_{i_d}\right) }{\D_{1\ldots 1}^{d-1}}
        \end{align*}
        and similarly
        \begin{equation*}
            \H_{i_1\ldots i_d} = \frac{\br^{(1)}\left(\lambda^{(1)}_{i_1}\right) \cdots \br^{(d)}\left(\lambda^{(d)}_{i_d}\right) }{\D_{1\ldots 1}^{d-1}}.
        \end{equation*}
        Hence, the separable rational function
        \begin{equation*}
            \tilde{\br}(\bz) = \frac{\br^{(1)}\left(z^{(1)}\right) \cdots \br^{(d)}\left(z^{(d)}\right) }{\D_{1\ldots 1}^{d-1}}
        \end{equation*}
        exactly recovers the data in $\D$. By considering the barycentric forms introduced in \eqref{eq:rjbarycentric}, we see that the barycentric form of $\tilde{\br}$ can be written as    
        \begin{align*}
             \tilde{\br}(\bz)&= \frac{1}{\D_{1\ldots 1}^{d-1}} \frac{\left(\cC^{(1)}\left(z^{(1)}\right)^\top \left( \alpha^{(1)} \circ \H^{(1)} \right) \right) \cdots \left( \cC^{(d)}\left(z^{(d)}\right)^\top \left( \alpha^{(d)} \circ \H^{(d)} \right)\right)}{ \left(\cC^{(1)}\left(z^{(1)}\right)^\top \alpha^{(1)} \right) \cdots \left( \cC^{(d)}\left(z^{(d)}\right)^\top \alpha^{(d)}\right)} \\
            &= \frac{1}{\D_{1\ldots 1}^{d-1}} \frac{\left[ \cC^{(1)}\left( z^{(1)} \right) \otimes \cdots \otimes \cC^{(d)}\left( z^{(d)} \right) \right]^\top \left[ \left(  \alpha^{(1)} \otimes \cdots \otimes \alpha^{(d)} \right) \circ \left( \H^{(1)} \otimes \cdots \otimes \H^{(d)} \right)\right]}{\left[ \cC^{(1)}\left( z^{(1)} \right) \otimes \cdots \otimes \cC^{(d)}\left( z^{(d)} \right) \right]^\top \left[ \left(  \alpha^{(1)} \otimes \cdots \otimes \alpha^{(d)} \right)\right]} \\
            &= \frac{\left[ \cC^{(1)}\left( z^{(1)} \right) \otimes \cdots \otimes \cC^{(d)}\left( z^{(d)} \right) \right]^\top \left[ \left(  \alpha^{(1)} \otimes \cdots \otimes \alpha^{(d)} \right) \circ \vectorize(\H) \right]}{\left[ \cC^{(1)}\left( z^{(1)} \right) \otimes \cdots \otimes \cC^{(d)}\left( z^{(d)} \right) \right]^\top \left[ \left(  \alpha^{(1)} \otimes \cdots \otimes \alpha^{(d)} \right)\right]}.
        \end{align*}
        This is exactly the barycentric form introduced in \eqref{eq:manyparambarycentric} but with a rank-$1$ tensor of barycentric coefficients
        \begin{equation*}
            \tilde{\alpha} = \alpha^{(1)} \otimes \cdots \otimes \alpha^{(d)}
        \end{equation*}
        instead of $\alpha$. Since the function $\tilde{\br}$ exactly recovers the data in $\D$ we get that the linearized LS errors along the lines of \eqref{eq:lserrorsum} are zero and thus
        \begin{equation*}
            \L_d \left(  \alpha^{(1)} \otimes \cdots \otimes \alpha^{(d)} \right) = 0.
        \end{equation*}

    \section{Proof of Lemma~\ref{lemma:blockseparablef}}
        \label{appendix:proofblockKlemma}
        The inverse of a matrix can be written as
		\begin{equation*}
			\bK(\bz)^{-1} = \frac{\operatorname{adj}(\bK(\bz))}{\det(\bK(\bz))},
		\end{equation*}
		where $\operatorname{adj}(\bK(\bz))$ denotes the adjugate which is a matrix whose entries are polynomials in terms of the entries of $\bK(\bz)$. Therefore, the representation
		\begin{equation*}
			\bff(\bz) = \bc^\top \bK(\bz)^{-1} \bb = \frac{\bc^\top \operatorname{adj}(\bK(\bz)) \bb}{\det(\bK(\bz))}
		\end{equation*}
		is a fraction of the polynomial $p(\bz) = \bc^\top \operatorname{adj}(\bK(\bz)) \bb$ in the numerator and $q(\bz) = \det(\bK(\bz))$ in the denominator. Next, we rewrite $\det(\bK(\bz))$ based on various determinant identities to further analyze the underlying structure: 
		\begin{align*}
			\det(\bK(\bz)) = \det\left(\begin{bmatrix}
				z^{(1)} \bK_{11} & \bK_{12} \\
				\bK_{21} & z^{(2)} \bK_{22}
			\end{bmatrix}\right) &= \det\left(z^{(2)} \bK_{22}\right) \det\left(z^{(1)}\bK_{11} - \frac{1}{z^{(2)}} \bK_{12} \bK_{22}^{-1} \bK_{21} \right).
		\end{align*}
		Since $\bK_{12},\bK_{21}$ have rank $r$ or less, we can write $\bK_{12} \bK_{22}^{-1} \bK_{21} = \bW \bY^\top$ where $\bW,\bY \in \C^{m_2 \times r}$. This allows for applying the low-rank update formula for determinants which yields
		\begin{align*}
			\det\left(z^{(1)} \bK_{11} - \frac{1}{z^{(2)}} \bW\bY^\top \right) &= \det\left(z^{(1)} \bK_{11}\right) \det\left(\bI_r - \frac{1}{z^{(1)}z^{(2)}} \bY^\top \bK_{11}^{-1} \bW \right) \\
			&= {z^{(1)}}^{m_1} \det\left(\bK_{11}\right) \frac{1}{{z^{(1)}}^r {z^{(2)}}^r} \det\left({z^{(1)}}^r {z^{(2)}}^r \bI_r - \bY^\top \bK_{11}^{-1}\bW\right)
		\end{align*}
		Note that $\det\left({z^{(1)}} {z^{(2)}} \bI_r - \bY^\top \bK_{11}^{-1}\bW\right)$ is a degree $r$ polynomial in ${z^{(1)}}{z^{(2)}}$ and thus has a separation rank of at most $r+1$. We therefore obtain that 
        \begin{equation*}
             \det(\bK(\bz)) =  \det\left(\bK_{11}\right) \det\left(\bK_{22}\right) {z^{(1)}}^{m_1-r} {z^{(2)}}^{m_2-r} \det\left({z^{(1)}}{z^{(2)}} \bI_r - \bY^\top \bK_{11}^{-1}\bW\right)
        \end{equation*}
       has a separation rank that is less than or equal to $r+1$.
	
\end{document}